\documentclass{amsart}

\usepackage[latin1]{inputenc}
\usepackage{amsmath, amssymb, mathabx, skak, graphicx, verbatim, color, textcomp, array, wasysym, cite, stmaryrd, MnSymbol, cancel, amsthm,}
\usepackage{hyperref}
\usepackage[caption=false, ...]{subfig}
\usepackage{pictexwd, dcpic}

	\theoremstyle{plain}
	\newtheorem{theorem}{Theorem}[section]
	\newtheorem{lemma}[theorem]{Lemma}
	\newtheorem{proposition}[theorem]{Proposition}
	
	\theoremstyle{definition}
	\newtheorem{definition}[theorem]{Definition}
	\newtheorem{example}[theorem]{Example}

\newcommand{\Av}{\operatorname{Av}}
\newcommand{\Inv}{\operatorname{Inv}}
\newcommand{\terms}{\operatorname{terms}}
\newcommand{\st}{\operatorname{st}}
\newcommand{\K}{\mathbb{K}}

\newcommand{\includesash}[2]{\raisebox{#1}{\includegraphics{#2}}}

\newcommand{\nobox}{\begin{picture}(5,15) \put(0,1){\textbardbl} \end{picture}}

\title[Hopf Algebra of Sashes]{Combinatorial Realization of the \\
Hopf Algebra of Sashes}
\author{Shirley Law}
\thanks{This research was partially funded by NSF grant DMS-1101568.}

\begin{document}

\begin{abstract}
A general lattice theoretic construction of Reading constructs Hopf subalgebras of the Malvenuto-Reutenauer Hopf algebra (MR) of permutations. 
The products and coproducts of these Hopf subalgebras are defined extrinsically in terms of the embedding in MR. 
The goal of this paper is to find an intrinsic combinatorial description of a particular one of these Hopf subalgebras.  
This Hopf algebra has a natural basis given by permutations that we call Pell permutations. 
The Pell permutations are in bijection with combinatorial objects that we call sashes, that is, tilings of a 1 by n rectangle with three types of tiles: black 1 by 1 squares, white 1 by 1 squares, and white 1 by 2 rectangles.  
The bijection induces a Hopf algebra structure on sashes. 
We describe the product and coproduct in terms of sashes, and the natural partial order on sashes. 
We also describe the dual coproduct and dual product of the dual Hopf algebra of sashes.
\end{abstract}
\maketitle

\section{Introduction}
\label{intro}

The focus of this research is on combinatorial Hopf algebras: Hopf algebras such that the basis elements of the underlying vector space are indexed by a family of combinatorial objects.
For each $n \geq 0$, let $O_n$ be a finite set of ``combinatorial objects''.
We define a graded vector space over a field $\K$, such that for each grade $n$ the basis vectors of the vector space are indexed by the elements of $O_n$.
That is, the graded vector space is: $\K[O_{\infty}]=\bigoplus_{{n}\ge0}\K[O_{n}]$.
For simplicity, we refer to a basis element of this vector space by the combinatorial object indexing it.
There is a more sophisticated approach for defining combinatorial Hopf algebras.
For more information see~\cite{AgBSottile}.

Let $S_{n}$ be the group of permutations of the set of the first $n$ integers $[{n}] = \left\{ 1, 2, \ldots, {n} \right\}$.  
Also define $[{n},{n}'] = \left\{ {n}, {n}+1, \ldots, {n}' \right\}$ for ${n}' \geq {n}$.  
For $x = x_1 x_2 \cdots x_{n} \in S_{n}$, an \emph{inversion} of $x$ is a pair $(x_i, x_j)$ where $i<j$ and $x_i>x_j$, and the \emph{inversion set of $x$} is the set of all such inversions.  
The weak order is the partial order on $S_{n}$ with $x\leq x'$ if and only if the inversion set of $x$ is contained in the inversion set of $x'$.   
The weak order is a lattice.  
The inverse $x^{-1}$ of a permutation $x \in S_{n}$ is the permutation $x^{-1}=y= y_1 \cdots y_n \in S_{n}$ such that $y_i = j$ when $x_j = i$.  

Let $T$ be a set consisting of integers $t_1 < t_2 < \cdots < t_{n}$.  
Given a permutation $x \in S_{n}$, the notation $( x )_{T}$ stands for the permutation of $T$ whose one-line notation has $t_j$ in the $i^{th}$ position when $x_i=j$.
On the other hand, given a permutation $x$ of $T$, the \emph{standardization}, $\st(x)$, is the unique permutation $y \in S_n$ such that $(y)_T=x$.

Now let $T$ be a subset of $[{n}]$.  For $x \in S_{n}$, the permutation $x|_T$ is the permutation of $T$ obtained by removing from the one-line notation for $x$ all entries that are not elements of $T$.

\begin{example} 
Let $x = 31254$, $T_1 = \{ 2,3,6,8,9 \}$, and $T_2 = \{ 2,3,5 \}$.  
Then, $(x)_{T_1} = 62398$ and thus $\st(62398)=31254$.
Also, $x|_{T_2} = 325$.
\end{example}

The Malvenuto-Reutenauer Hopf algebra MR is a graded Hopf Algebra denoted by $(\K[S_\infty],\bullet,\Delta)$.  
Let $\K[S_{\infty}]=\bigoplus_{{n}\ge0}\K[S_{n}]$ be  a graded vector space.  
Let $x = x_1 x_2 \cdots x_p \in S_p$ and $y = y_1 y_2 \cdots y_q \in S_q$.  
Define $y' = y'_1 \cdots y'_q$ to be $( y )_{[p+1,p+q]}$  so that $y'_i = y_i+p$.  
A \emph{shifted shuffle} of $x$ and $y'$ is a permutation $z \in S_{n}$ where $n=p+q$, $z|_{[p]} = x$ and $z|_{[p+1,n]} = y'$.  
The product of $x$ and $y$ in MR is the sum of all the shifted shuffles of $x$ and $y$.
Equivalently, 
\begin{equation}
x \bullet y = \sum [x\cdot y', y'\cdot x]
\end{equation}
where $x\cdot y'$ is the concatenation of the permutations $x$ and $y'$, and $\sum[x\cdot y', y'\cdot x]$ denotes the sum of all the elements in the weak order interval $[x\cdot y', y'\cdot x]$.
A paper by Loday and Ronco~\cite{LR} introduces the representation of Hopf algebra operations in terms of sums over weak order intervals.  
The coproduct in MR is:
\begin{equation}
\Delta(x) = \sum_{i=0}^p \st(x_1\cdots x_i) \otimes \st(x_{i+1}\cdots x_p)
\end{equation}
where $\st(x_1\cdots x_0)$ and $\st(x_{p+1}\cdots x_p)$ are both interpreted as the empty permutation $\emptyset$.

Define the map $\Inv: S_{n} \to S_{n}$ by $\Inv(x) = x^{-1}$ and extend the map linearly to a map $\Inv: \K S_{\infty} \to \K S_{\infty}$.  
MR is known to be self dual~\cite{mr} and specifically $\Inv$ is an isomorphism from $(\K[S_\infty],\bullet,\Delta)$ to the graded dual Hopf algebra $(\K[S_\infty],\Delta^*, m^*)$.  
Let $x \in S_p$, $y \in S_q$, and $z \in S_{n}$, where $p+q=n$.  
Given a subset $T$ of $p$ elements of $[n]$, $T^C$ denotes the complement of $T$ in $[n]$.  
The dual product is given by:
\begin{equation} 
\label{MRdualprod} 
\Delta^*(x \otimes y) = \Inv(x^{-1} \bullet y^{-1}) = \sum_{\substack{T \subseteq [n], \\ |T|=p}} (x)_T \cdot (y)_{T^C}, 
\end{equation}
and the dual coproduct is:
\begin{equation} 
\label{MRdualcoprod} 
m^*(z) = (\Inv \otimes \Inv)(\Delta(z^{-1})) = \sum_{i=0}^{n} z|_{[i]} \otimes \st(z|_{[i+1,n]}) 
\end{equation}
where $z|_{[0]}$ and $z|_{[n+1,n]}$ are both interpreted as the empty permutation $\emptyset$.

Now that we have explicitly described both the Hopf algebra of permutations and the dual Hopf algebra of permutations, we will present a family of Hopf subalgebras that are defined by a particular pattern-avoidance condition.  
This family of Hopf algebras is defined by Reading~\cite{latcong}.

For some $k \geq 2$, let $V \subseteq [2, k-1]$ such that $|V|=j$ and let $V^C$ be the complement of $V$ in $[2, k-1]$.  
A permutation $x \in S_{n}$ \emph{avoids} the pattern $V(k1)V^C$ if for every subsequence $x_{i_1} x_{i_2} \cdots x_{i_k}$ of $x$ with $i_{j+2} = i_{j+1}+1$, the standardization $\st(x_{i_1} x_{i_2} \cdots x_{i_k})$ is not of the form $v(k1)v'$ for any permutation $v$ of the set $V$ and any permutation $v'$ of $V^C$.  
In the notation of Babson and Steingrimsson~\cite{permpat} avoiding $V(k1)V^C$ means avoiding all patterns of the form $v_1 - \cdots - v_j - k 1 - v'_1 - \cdots - v'_{k-j-2}$, where $v_1 \cdots v_j$ is a permutation of $V$ and $v'_1 \cdots v'_{k-j-2}$ is a permutation of $V^C$.

Let $U$ be a set of patterns of the form $V(k1)V^C$, where $|V|$ and $k$ can vary.  
Define $\Av_n$ to be the set of permutations in $S_n$ that avoid all of the patterns in $U$. 
We define a graded Hopf algebra $(\K[\Av_\infty],\bullet_{\Av},\Delta_{\Av})$ as a graded Hopf subalgebra of MR.  
Let $\K[\Av_n]$ be a vector space, over a field $\K$, with basis vectors indexed by the elements of $\Av_{n}$, and let $\mathbb{K}[\Av_{\infty}]$ be the graded vector space $\bigoplus_{{n}\ge0}\K[\Av_{n}]$.  
The product and coproduct on $\K[\Av_\infty]$ are described below.

We define a map $\pi_\downarrow: S_n \to \Av_n$ recursively.
If $x\in \Av_n$ then define $\pi_\downarrow(x)=x$.
If $x \in S_{n}$, but $x \notin \Av_{n}$, then $x$ contains an instance of a pattern $V(k1)V^C$ in $U$.  
That is, there exists some subsequence $x_{i_1} x_{i_2} \cdots x_{i_k}$ of $x$, where $i_{j+2} = i_{j+1}+1$ and $j = |V|$, such that $st(x_{i_1} x_{i_2} \cdots x_{i_k}) = v k 1 v'$ for some permutations $v$ and $v'$ of $V$ and $V^C$.  
Exchange $x_{i_{j+1}}$ and $x_{i_{j+2}}$ in $x$ to create a new permutation $x'$, calculate $\pi_\downarrow(x')$ recursively and set $\pi_\downarrow(x)=\pi_\downarrow(x')$.
The recursion must terminate because an inversion of $x$ is destroyed at every step, and because the identity permutation is in $\Av_n$.
The map $\pi_\downarrow$ is well-defined as explained in~\cite[Remark 9.5]{latcong}.  We emphasize that the definition of $\pi_\downarrow$ is dependent on $U$.  

The map $\pi_\downarrow$ defines an equivalence relation with permutations $x, x' \in S_{n}$ equivalent if and only if $\pi_\downarrow(x) = \pi_\downarrow(x')$.  
The set $\Av_{n}$ is a set of representatives of these equivalence classes. 
This equivalence relation is a lattice congruence on the weak order.
Therefore the poset induced on $\Av_n$ by the weak order is a lattice (also denoted by $\Av_n$) and the map $\pi_\downarrow$ is a lattice homomorphism from the weak order to $\Av_n$.
The congruence classes defined by $\pi_\downarrow$ are intervals, and $\pi_\downarrow$ maps an element to the minimal element of its congruence class.
Let $\pi^\uparrow$ be the map that takes an element to the maximal element of its congruence class.

The following proposition is a special case of \cite[Proposition~2.2]{latcong}.  
The congruence on $S_n$ defined by $\pi_\downarrow$ is denoted by $\Theta$.
For $x \in S_n$, the congruence class of $x$ mod $\Theta$ is denoted by $[x]_\Theta$.

\begin{proposition}
\label{latcong2.2}
Given $S_n$ a finite lattice, $\Theta$ a congruence on $S_n$, and $x \in S_n$, the map $y \to [y]_\Theta$ restricts to a one-to-one correspondence between elements of $S_n$ covered by $\pi_\downarrow(x)$ and elements of $\Av_n$ covered by  $[x]_\Theta$.
\end{proposition}

Both $\pi_\downarrow$ and $\pi^\uparrow$ are order preserving and $\pi^\uparrow\circ\pi_\downarrow=\pi^\uparrow$ and $\pi_\downarrow\circ\pi^\uparrow=\pi_\downarrow$.
A \emph{$\pi_\downarrow$-move} is the result of switching two adjacent entries of a permutation in the manner described above.  
That is, it changes $\cdots k1 \cdots$ to $\cdots 1k \cdots$ for some pattern in $U$.
A \emph{$\pi^\uparrow$-move} is the result of switching two adjacent entries of a permutation in a way such that a $\pi^\uparrow$-move undoes a $\pi_\downarrow$-move.  
That is, it changes $\cdots 1k \cdots$ to $\cdots k1 \cdots$.  

We define a map $r:\K[S_\infty]\to\K[\Av_\infty]$ that identifies the representative of a congruence class.  
Given $x\in S_{n}$,
\[r(x)=\begin{cases}
x&\mbox{if }x\in\Av_{n}\\
0&\mbox{otherwise}.
\end{cases}
\]

Similarly, we define a map $c:\K[\Av_{\infty}]\to\K[S_{\infty}]$ that takes an avoider to the sum of its congruence class:
\[c(x)=\sum_{\substack{y \text{ such that} \\ \pi_\downarrow(y)=x}}y.\]

We now describe the product and coproduct in $(\K[\Av_{\infty}], \bullet_{\Av}, \Delta_{\Av})$.
Let $x \in \Av_{p}$, and let $y \in \Av_{q}$.  
Then:
\begin{equation} 
\label{AVprod} 
m_{\Av}(x \otimes y) = x \bullet_{\Av} y=r(x \bullet y). 
\end{equation}

Just as the product in MR is $ \sum [x \cdot y', y' \cdot x]$, we can view this product as:

\begin{equation} 
\label{AVprod2} 
x\bullet_{\Av}y = \sum [x \cdot y', \pi_\downarrow(y' \cdot x)],
\end{equation}
where $[x \cdot y', \pi_\downarrow(y' \cdot x)]$ is an interval on the lattice $\Av_n$.

The coproduct is:
\begin{equation} 
\label{AVcoprod} 
\Delta_{\Av}(z)=(r\otimes r)(\Delta(c(z))). 
\end{equation}

We now describe the Hopf algebra denoted by $(\K[\Av_\infty], \Delta^*_{\Av}, \bullet^*_{\Av})$ that is dual to $(\K[\Av_\infty],\bullet_{\Av},\Delta_{\Av})$.  
We extend the map $\pi_\downarrow$ linearly, so $\pi_\downarrow$ is a map from $\K[S_{\infty}]$ to $\K[\Av_{\infty}]$.
The map that is dual to the map $c$ is $c^*:\K[S_{\infty}]\to\K[\Av_{\infty}]$, where $c^*(x) = \pi_\downarrow(x)$ for $x \in \K[S_\infty]$.  
The map that is dual to the map $r$ is $r^*:\K[\Av_\infty]\to\K[S_\infty]$, where $r^*(x) = x$ for $x \in \K[\Av_\infty]$.

Let $z \in \Av_{n}$, where $n=p+q$.
The dual coproduct is given by dualizing Equation~\eqref{AVprod}, so that:
\begin{equation} 
\label{AVdualcoprod} 
m^*_{\Av}(z) =m^*(z). 
\end{equation}

The dual product $\Delta^*_{\Av}$ is given by dualizing Equation~\eqref{AVcoprod}:
\begin{equation} 
\label{AVdualprod} 
\Delta^*_{\Av}(x \otimes y) = \pi_\downarrow \Delta^*(x \otimes y). 
\end{equation}
Combining Equation~\eqref{AVdualprod} with Equation~\eqref{MRdualprod}, we have:
\begin{equation} 
\label{AVdualprod2} 
\Delta^*_{\Av}(x \otimes y) = \sum_{\substack{T\subseteq [n] \\ |T|=p}} \pi_\downarrow((x)_T \cdot (y)_{T^C}) 
\end{equation}

Equation~\eqref{AVdualprod2} leads to the following order theoretic description of the coproduct $\Delta_{\Av}$, which was worked out jointly with Nathan Reading.

Given $z\in\Av_n$, a subset $T\subseteq[n]$ is \emph{good with respect to $z$} if there exists a permutation $z'=z'_1\cdots z'_n$ with $\pi_\downarrow(z')=z$ such that $T=\{ z'_1,\ldots,z'_{|T|}\}$.
Suppose $T$ is good with respect to $z$, let $p=|T|$ and let $q=n-p$.
Let $z_{\min}$ be minimal, in the weak order on $S_n$, among permutations equivalent to $z$ and whose first $p$ entries are the elements of $T$.
Let $z_{\max}$ be maximal, in the weak order, among such permutations.
Define $I_T$ to be the sum over the elements in the interval $[\st( z_{\min}|_{T} ),\pi_\downarrow\st( z_{\max}|_{T} )]$ in $\Av_p$ and define $J_T$ to be the sum over the elements in the interval $[\st( z_{\min}|_{T^C} ),\pi_\downarrow\st( z_{\max}|_{T^C} )]$ in $\Av_q$.

\begin{theorem}\label{coprod int}
Let $z \in \Av_n$.
Then 
\[\Delta_{\Av}(z) = \sum_{T \text{ is good}} I_T \otimes J_T\]
where $I_T = \sum [\st( z_{\min}|_{T} ),\pi_\downarrow\st( z_{\max}|_{T} )]$, $J_T = \sum [\st( z_{\min}|_{T^C} ),\pi_\downarrow\st( z_{\max}|_{T^C} )]$.
\end{theorem}

To prove Theorem~\ref{coprod int}, we first need several lemmas.

\begin{lemma}
The elements in the interval $[z_{\min}, z_{\max}]$ are equivalent to $z$ and their first $p$ entries are the elements of $T$.
\end{lemma}

\begin{proof}
All of the elements in the interval $[z_{\min}, z_{\max}]$ are equivalent to $z$ because equivalence classes are intervals in the weak order.  
To prove the rest of the lemma, suppose for the sake of contradiction that there is an element $z' \in [z_{\min}, z_{\max}]$ whose first $p$ entries are not the elements of $T$.  
That is, $z'$ has some $y \in T^C$ before some $x \in T$.  
If $x < y$, then $(y,x) \in \Inv(z')$, but $(y, x) \notin \Inv(z_{\max})$, so $z' \notin [z_{\min}, z_{\max}]$.  
If $y < x$, then $(x, y) \in \Inv(z_{\min})$, but $(x, y) \notin \Inv(z')$, so $z' \notin [z_{\min}, z_{\max}]$.  
Therefore the first $p$ entries of elements in the interval $[z_{\min}, z_{\max}]$ are the elements of $T$.
\end{proof}

\begin{lemma}  \label{coprod convex} 
Suppose $T \subseteq [n]$ with $|T|=p$.  Let $q=n-p$.  Suppose also that $x_1 \leq x_2 \leq x_3$ in $\Av_p$, and that $y_1 \leq y_2 \leq y_3$ in $\Av_q$.  
If $\pi_\downarrow ( (x_1)_T \cdot (y_1)_{T^C} ) = \pi_\downarrow ( (x_3)_T \cdot (y_3)_{T^C} )  = z$, then $\pi_\downarrow ( (x_2)_T \cdot (y_2)_{T^C} ) = z$.
\end{lemma}

\begin{proof}  
If $x_1 \leq x_2 \leq x_3$, and $y_1 \leq y_2 \leq y_3$, then $(x_1)_T \cdot (y_1)_{T^C} \leq (x_2)_T \cdot (y_2)_{T^C} \leq (x_3)_T \cdot (y_3)_{T^C}$.  
Since $\pi_\downarrow$ is an order preserving map, $\pi_\downarrow ( (x_1)_T \cdot (y_1)_{T^C} ) \leq \pi_\downarrow ( (x _2)_T \cdot (y_2)_{T^C} ) \leq \pi_\downarrow ( (x _3)_T \cdot (y_3)_{T^C} )$.  
The assertion of the lemma follows.
\end{proof}

\begin{lemma} \label{meets_joins} 
Suppose $x_1, x_2 \in S_p$ and $y_1, y_2 \in S_q$.  
Suppose $T \subseteq [n]$, where $n=p+q$, and with $|T|=p$.  The following identities hold:
\begin{center}
$(x_1)_T \cdot (y_1)_{T^C} \vee (x_2)_T \cdot (y_2)_{T^C} = (x_1 \vee x_2)_T \cdot (y_1 \vee y_2)_{T^C}$\\
$(x_1)_T \cdot (y_1)_{T^C} \wedge (x_2)_T \cdot (y_2)_{T^C} = (x_1 \wedge x_2)_T \cdot (y_1 \wedge y_2)_{T^C}$
\end{center}
\end{lemma}

\begin{proof}
First we consider the identity with joins.
There are three different kinds of inversions in $\displaystyle{(x_1)_T \cdot (y_1)_{T^C}}$: inversions within $x_1$, inversions within $y_1$, and inversions between $T$ and $T^C$.
The inversion set of the permutation on the left hand side of the equation is the union of: inversions within $x_1$ in terms of $T$, inversions within $y_1$ in terms of $T^C$, inversions within $x_2$ in terms of $T$, inversions within $y_2$ in terms of $T^C$, and inversions between $T$ and $T^C$.
Similarly, the inversion set of the permutation on the right hand side of the equation is the union of: inversions within $x_1$ or $x_2$ in terms of $T$, inversions within $y_1$ or $y_2$ in terms of $T^C$, and inversions between $T$ and $T^C$.
Therefore the permutation on the left hand of the equation and the permutation on the right hand of the equation have identical inversion sets and are thus the same.

The proof for the identity with meets is identical except for examining intersections of the inversion sets instead of unions.
\end{proof}

\begin{proof}[Proof of Theorem~\ref{coprod int}]
In light of Equation~\eqref{AVdualprod2}, $\Delta_{\Av}(z)$ is the sum, over $T \subseteq [n]$, of terms $x \otimes y \in \Av_p \otimes \Av_q$ such that $\pi_\downarrow( (x)_T \cdot (y)_{T^C} ) =z$.  
Some terms $x \otimes y$ may appear in $\Delta_{\Av}(z)$ with coefficient greater than 1, but for each $T$, a term $x \otimes y$ occurs at most once.
Let $\terms(z,T)$ be the set $\{ x \otimes y : \pi_\downarrow( (x)_T \cdot (y)_{T^C} )=z\}$.
It is immediate that when $\terms(z,T)$ is nonempty, $T$ is good with respect to $z$.
On the other hand, if $T$ is good with respect to $z$, then let $z'$ have $\pi_\downarrow(z')=z$ and $\{z'_1, z'_2, \dots, z'_{|T|}\}=T$.
Let $x\in S_p$ and $y \in S_q$ be such that $z'=(x)_T\cdot (y)_{T^C}$. 
Then $\pi_\downarrow(x) \in \Av_p$ and $\pi_\downarrow(y) \in \Av_q$.
Since $\pi_\downarrow(x)$ is obtained from $x$ by a sequence of $\pi_\downarrow$-moves, and $\pi_\downarrow(y)$ is obtained similarly from $y$, we see that $(\pi_\downarrow(x))_T\cdot(\pi_\downarrow(y))_{T^C}$ is obtained from $z'=(x)_T\cdot (y)_{T^C}$ by a sequence of $\pi_\downarrow$-moves.  
Thus, $\pi_\downarrow((\pi_\downarrow(x))_T\cdot(\pi_\downarrow(y))_{T^C}) = \pi_\downarrow(z')=z$, so $\pi_\downarrow(x)\otimes \pi_\downarrow(y) \in \terms(z,t)$ and in particular $\terms(z,T)$ is nonempty.

Next, we need to show that, for each good subset $T$, the set $\terms(z,T)$ is of the form $I_T \otimes J_T$.  
For convenience, we consider each $x \otimes y$ as an element of $\Av_p \times \Av_q$ without rewriting $x \otimes y$ as $(x,y)$.

Suppose $x_1 \otimes y_1$ and $x_2 \otimes y_2$ are in $\terms(z,T)$.  
Then by Lemma~\ref{meets_joins}, 
\[\pi_\downarrow \big((x_1 \vee x_2)_T \cdot (y_1 \vee y_2)_{T^C}\big) = \pi_\downarrow \big((x_1)_T \cdot (y_1)_{T^C} \vee (x_2)_T \cdot (y_2)_{T^C} \big).\]  
Since $\pi_\downarrow$ is a lattice homomorphisim, the latter is 
\[\pi_\downarrow \big((x_1)_T \cdot (y_1)_{T^C} \big) \vee \pi_\downarrow \big((x_2)_T \cdot (y_2)_{T^C} \big) = z \vee z = z.\]  
Thus $(x_1 \vee x_2) \otimes (y_1\vee y_2)$ is in $\terms(z,T)$.
The same argument holds for meets, so $\terms(z,T)$ is closed under meets and joins in the product order $\Av_p \times \Av_q$.  
Lemma~\ref{coprod convex} implies that $\terms(z,T)$ is order-convex in $\Av_p \times \Av_q$.
An order-convex subset that is closed under meets and joins is necessarily an interval.

Suppose $x\otimes y < \st ( z_{\min}|_{T} ) \otimes \st ( z_{\min}|_{T^C} )$ in $\Av_p\times\Av_q$.
Then $(x)_T \cdot (y)_{T^C} < z_{\min}$ in $S_n$.
Thus $\pi_\downarrow \big((x)_T \cdot (y)_{T^C} \big)\neq z$, by the definition of $z_{\min}$, and therefore $x\otimes y\not\in\terms(z,T)$.
Thus $\st( z_{\min}|_{T} )\otimes\st( z_{\min}|_{T^C} )$ is the minimal element of $\terms(z,T)$.

Now suppose $x \otimes y > \pi_\downarrow\st( z_{\max}|_{T} ) \otimes \pi_\downarrow\st( z_{\max}|_{T^C} )$ in $\Av_p \times \Av_q$.
Then since $\pi^\uparrow$ is order-preserving and $\pi^\uparrow\circ\pi_\downarrow=\pi^\uparrow$, we see that $\pi^\uparrow(x)\otimes \pi^\uparrow(y) > \pi^\uparrow\st( z_{\max}|_{T} )\otimes\pi^\uparrow\st( z_{\max}|_{T^C} )$.
Thus on the lattice $S_n$, 
\begin{equation}\label{coprod int eq}
\big(\pi^\uparrow(x) \big)_T \cdot \big(\pi^\uparrow(y) \big)_{T^C}>\big(\pi^\uparrow\st( z_{\max}|_{T} )\big)_T\cdot\big(\pi^\uparrow\st( z_{\max}|_{T^C} )\big)_{T^C}.
\end{equation}
The right side of Equation~\eqref{coprod int eq} is obtained from $z_{\max}$ by standardizing the first part of the permutation, doing some $\pi^\uparrow$-moves, unstandardizing, and then repeating for the last part of the permutation.
The same result can be obtained by simply applying the corresponding $\pi^\uparrow$-moves to $z_{\max}$, without standardizing and unstandardizing.
In particular, the right side of Equation~\eqref{coprod int eq} is greater than or equal to $z_{\max}$.
Now Equation~\eqref{coprod int eq} implies that $\big(\pi^\uparrow(x)\big)_T \cdot \big(\pi^\uparrow(y)\big)_{T^C}$ is strictly greater than $z_{\max}$.
The definition of $z_{\max}$ says that $\pi^\uparrow(x)\otimes \pi^\uparrow(y)\not\in\terms(z,T)$.
Thus $\big(\pi^\uparrow(x)\big)_T \cdot \big(\pi^\uparrow(y)\big)_{T^C}$ is not equivalent to $z$.

But $\big(\pi^\uparrow(x)\big)_T \cdot \big(\pi^\uparrow(y)\big)_{T^C}$ is obtained from $(x)_T\cdot(y)_{T^C}$ by standardizing the first part, doing some $\pi^\uparrow$-moves, unstandardizing, and then repeating for the last part.
The same result is again obtained by simply applying the the corresponding $\pi^\uparrow$-moves to $(x)_T\cdot(y)_{T^C}$, without standardizing and unstandardizing.
Thus $\big(\pi^\uparrow(x)\big)_T \cdot \big(\pi^\uparrow(y)\big)_{T^C}$ is equivalent to $(x)_T\cdot(y)_{T^C}$, which is therefore not equivalent to $z$.
We have shown that $x\otimes y\not\in\terms(z,T)$.

Thus, we have shown that $\terms(z,T)$ equals
\[[\st( z_{\min}|_{T} )\otimes\st( z_{\min}|_{T^C} ),\pi_\downarrow\st( z_{\max}|_{T} )\otimes\pi_\downarrow\st( z_{\max}|_{T^C} )].\]
Any interval in $\Av_p\times\Av_q$ is the product of an interval in $\Av_p$ with an interval in $\Av_q$.
Thus $\terms(z,T)$ is $I_T\otimes J_T$.
\end{proof}

The proof of Theorem~\ref{coprod int} also establishes the following more detailed statement.

\begin{proposition}
\label{coprod dualprod}
For some $T \subseteq [n]$, $x \otimes y\in \terms(z,T)$ if and only if $x \otimes y$ is a term of $I_T \otimes J_T$ in $\Delta_{\Av}(z)$. 
\end{proposition}

\begin{proof}
Since $x \otimes y\in \terms(z,T)$ means that $\pi_\downarrow \big((x)_T \cdot (y)_{T^C} \big) = z$, we see from Equation\eqref{AVdualprod2} and Theorem~\ref{coprod int} that for a fixed set $T$, $x \otimes y$ is a term of the summand indexed by $T$ in $\Delta_{\Av}(z)$ if and only if $z$ is the summand indexed by $T$ in $\Delta^*_{\Av}(x \otimes y)$.
\end{proof}

\section{Pell Permutations and Sashes}
\label{objects}

Given a permutation $x=x_1 x_2 \cdots x_{n} \in S_{n}$, for each $i \in [n-1]$, there is a nonzero integer $j$ such that $x_i = x_{i+1} + j$.  
If $j>0$, then there is an \emph{descent} of size $j$ in the $i^{th}$ position of $x$.  
A Pell permutation is a permutation of $[n]$ with no descents of size larger than $2$, and such that for each descent $x_i=x_{i+1}+2$, the element $x_{i+1}+1$ is to the right of $x_{i+1}$.
We write $P_n$ for the set of Pell permutations in $S_n$. 

Let us consider how many Pell permutations of length ${n}$ there are.  
Given $x \in P_{{n}-1}$, we can place ${n}$ at the end of $x$ or before ${n}-1$.  
We can also place ${n}$ before ${n}-2$, but only if ${n}-1$ is the last entry of $x$.
(If ${n}-1$ is to the right of ${n}-2$ and there is some entry $i$ to the right of ${n}-1$; then ${n}-2$, ${n}-1$, $i$ form a $231$ pattern.)  
Therefore $|P_{n}| = 2 |P_{{n}-1}| + |P_{{n}-2}|$.  
This recursion, with the initial conditions $|P_0|=0$ and $|P_1|=1$, defines the Pell numbers as defined by~\cite[Sequence A000129]{OEIS}.

\begin{lemma} \label{lem:pell def}
$P_{n} = \Av_{n}$ for $U=\{2(31), (41)23\}$.
\end{lemma}

\begin{proof}
Suppose $x \in P_{n}$.  
Since $x$ does not have any descents larger than $2$, it avoids $(41)23$.
For each descent $x_i=x_{i+1}+2$ in $x$, the element $x_{i+1}+1$ is to the right of $x_{i+1}$.
Thus $x$ also avoids $2(31)$.
Now suppose $x \in \Av_{n}$.  
Suppose $x$ has a descent $x_i=x_{i+1}+j$.  
Because $x$ avoids $2(31)$, the entries $x_{i+1}+1, ... , x_{i+1}+j-1$ are to the right of the $x_{i+1}$.  
Thus, since $x$ avoids $(41)23$ we see that $j \leq 2$ and conclude that $x \in P_{n}$.
\end{proof}

The poset induced on $P_n$ by the weak order is a lattice (also denoted by $P_n$).
As a consequence of Lemma~\ref{lem:pell def}, there is a Hopf algebra $(\K[\Av_\infty], \bullet_{\Av}, \Delta_{\Av})$ of Pell permutations. 
For the rest of this paper we fix $U=\{2(31), (41)23\}$.   

There is a combinatorial object in bijection with Pell permutations that will allow us to have a more natural understanding of the Hopf algebra of Pell permutations.  

A \emph{sash} of length $n$ is a tiling of a $1 \times n$ rectangle by black $1 \times 1$ squares, white $1 \times 1$ squares, and/or white $1 \times 2$ rectangles.   
The set of sashes of length $n$ is called $\Sigma_{n}$.
There are no sashes of length -1 so $\Sigma_{-1} = \emptyset$, and there is one sash of length 0, a 1 by 0 rectangle denoted \nobox, so $|\Sigma_0|=1$.
There are two sashes of length $1$: \includesash{-1.5pt}{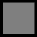} and \includesash{-1.5pt}{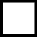}. 
The five sashes of length $2$ and the twelve sashes of length $3$ are shown in Figure~\ref{fig: Sigma 3&4}.
The poset structure of these sashes will be explained later in this section.

\begin{figure}
	\centering
	\begin{minipage}{.2\textwidth}
	\begin{center}
	\includegraphics[scale=1]{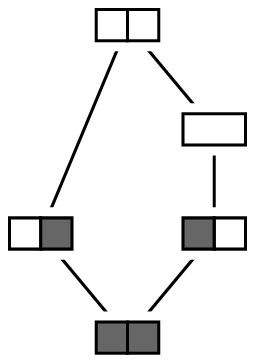}
	\end{center}
	\end{minipage}
	\qquad
	\begin{minipage}{.55\textwidth}
	\begin{center}
	\includegraphics[scale=.6]{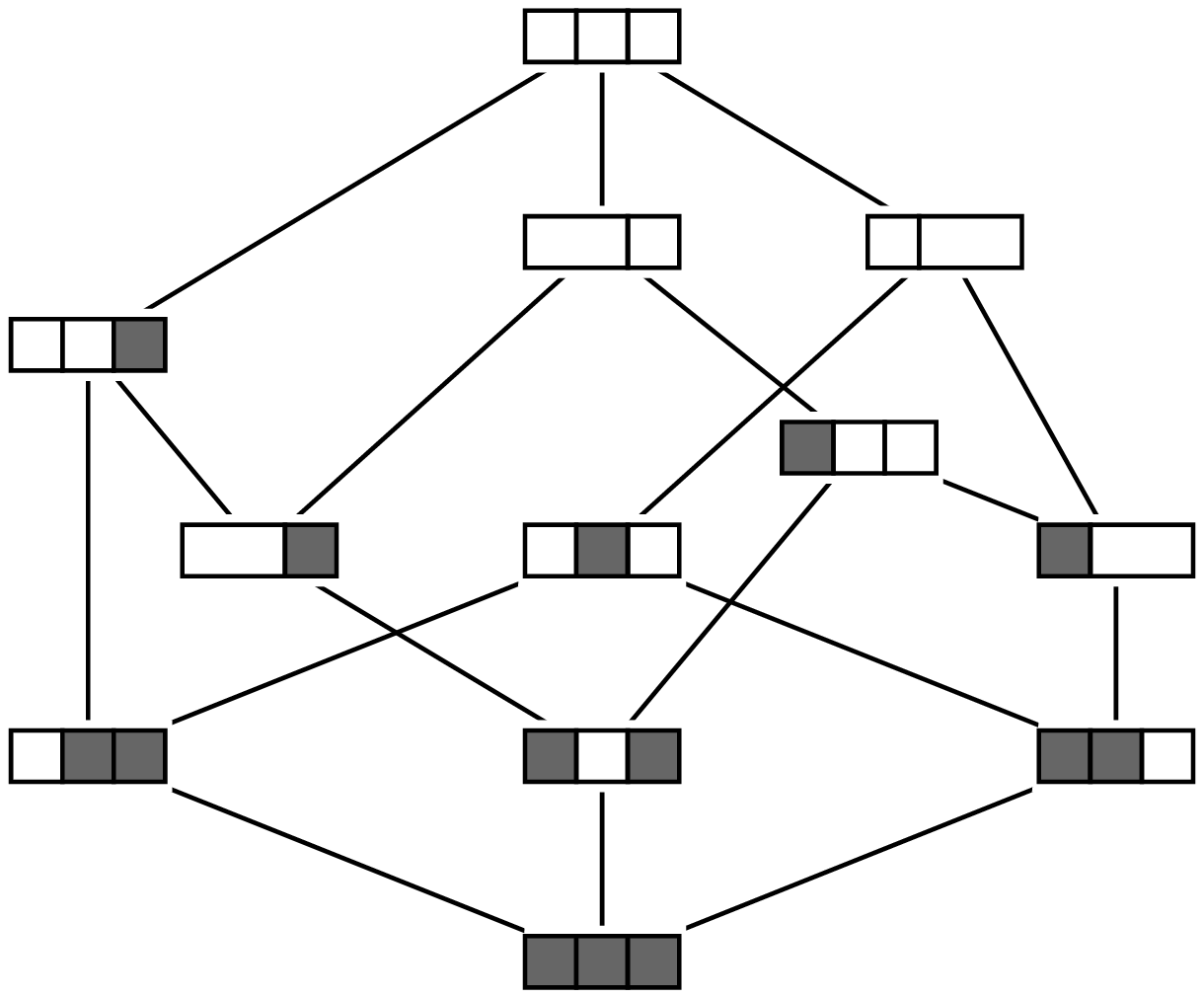}
	\end{center}
	\end{minipage}
	\caption{The elements of $\Sigma_3$ (left) and $\Sigma_4$ (right).}
	\label{fig: Sigma 3&4}
\end{figure}
		
A sash of length $n$ starts with either a black square, a white square, or a rectangle.
Thus $|\Sigma_{n}| = 2 |\Sigma_{n-1}| + |\Sigma_{n-2}|$.  
Since $|\Sigma_{-1}|=0$ and $|\Sigma_0|=1$, there is a bijection between Pell permutations of length $n$ and sashes of length $n-1$.
We now describe a bijection that we use to induce a Hopf Algebra structure on sashes.

\begin{definition}  
We define a map $\sigma$ from $S_{n}$ to $\Sigma_{n-1}$. 
Let $x \in S_{n}$.  
We build a sash $\sigma(x)$ from left to right as we consider the entries in $x$ from $1$ to $n-1$.
For each value $i \in [n-1]$, if $i+1$ is to the right of $i$, place a black square on the sash, and if $i+1$ is to the left of $i$, place a white square on the sash. 
There is one exception: If $i+1$ is to the right of $i$, and $i+2$ is to the left of $i$ (and of $i+1$), then place a rectangle in the $i^{th}$ and $(i+1)^{st}$ positions of the sash.
We also define $\sigma(1) = \nobox$ and $\sigma(\emptyset) = \emptyset$.
\end{definition}

From the definition of the map $\sigma$ we see that $\sigma$ sometimes involves replacing an adjacent black square and white square by a rectangle.  
Later, we will sometimes break a rectangle into a black square and a white square.

\begin{example}
Here is the procedure for computing $\sigma(421365)$.
\begin{center}  
\begin{tabular}{l l l}
$2$ is to the left of $1$ & $\rightarrow$ & \includesash{-1.5pt}{w} \\
$3$ is to the right of $2$ & $\rightarrow$ & \includesash{-1.5pt}{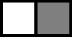} \\
$4$ is to the left of $3$ and also to the left of $2$ & $\rightarrow$ &  \includesash{-1.5pt}{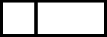} \\
$5$ is to the right of $4$ & $\rightarrow$ & \includesash{-1.5pt}{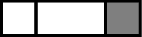} \\
$6$ is to the left of $5$ but to the right of $4$ & $\rightarrow$ &  \includesash{-1.5pt}{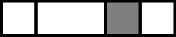} \\
\end{tabular}
\end{center}
\end{example}

Let $T$ be a set of $n$ integers and let $x$ be a permutation of $T$.  
We define $\sigma(x) = \sigma(\st(x))$.

\begin{example}
$\sigma(742598) = \sigma(\st(742598)) = \sigma(421365) = \includesash{-1.5pt}{wrbw}$
\end{example}

\begin{definition} 
We define a map $\eta:\Sigma_{n-1}\to P_{n}$.
To calculate $\eta(A)$ for a sash $A \in \Sigma_{n-1}$, we place the numbers $1$ through $n$ one at a time.
Place the number $1$ to begin and let $i$ run from $1$ to $n-1$.  
If $A$ has either a black square or the left half of a rectangle in the $i^{th}$ position, place $i+1$ at the right end of the permutation. 
If $A$ has either a white square or the right half of a rectangle in the $i^{th}$ position, place $i+1$ immediately to the left of $i$ or $i-1$ respectively. 
We also define $\eta(\nobox) = 1$ and $\eta(\emptyset) = \emptyset$.
\end{definition}

It is immediate that this construction yields a Pell permutation because the output has no descents of size larger than $2$, and for each descent of size $2$, the value in between the values of the descent is to the right of the descent.

\begin{example}
Here are the steps to calculate $\eta(A)$ for $A = \includesash{-1.5pt}{wrbw}$.
\begin{center}
\begin{tabular}{l l l}
 & $\rightarrow$ & 1 \\
\includesash{-1.5pt}{w} & $\rightarrow$ & 21 \\
\includesash{-1.5pt}{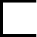} & $\rightarrow$ & 213 \\
\includesash{-1.5pt}{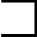} & $\rightarrow$ & 4213 \\
\includesash{-1.5pt}{b} & $\rightarrow$ & 42135 \\
\includesash{-1.5pt}{w} & $\rightarrow$ & 421365 \\
\end{tabular} 
\end{center}
\end{example}

\begin{theorem}
The restriction of $\sigma$ to the Pell permutations is a bijection $\sigma:P_{n} \to \Sigma_{n-1}$ whose inverse is given by $\eta:\Sigma_{n-1} \to P_{n}$.
\end{theorem}

\begin{proof}
Let $A \in \Sigma_{n-1}$.
We first show that $\sigma(\eta(A))=A$.  
If $A$ has a black square in position $i$, then $\eta(A)$ has $i+1$ to the right of $i$ and $i+2$ not to the left of $i$.
So $\sigma(\eta(A))$ also has a black square in the $i^{th}$ position.
If $A$ has a white square in position $i$, then $\eta(A)$ has $i+1$ immediately to the left of $i$.
So $\sigma(\eta(A))$ also has a white square in the $i^{th}$ position.
If $A$ has a rectangle in positions $i$ and $i+1$, then $\eta(A)$ has $i+1$ to the right of $i$, and $i+2$ immediately to the left of $i$.  
So $\sigma(\eta(A))$ also has a rectangle in the $i^{th}$ and $(i+1)^{st}$ positions.
We conclude that $\sigma(\eta(A))=A$.

We have constructed a Pell permutation $\eta(A)$ that maps to $A$ under $\sigma$, therefore $\sigma$ is surjective.
Since we know $|P_n|=|\Sigma_{n-1}|$, the map $\sigma$ restricted to Pell permutations is a bijection.
The inverse map of $\sigma$ is $\eta$.
\end{proof}

\begin{proposition}
\label{sigma_equiv} 
$x, y \in S_{n}$ are equivalent if and only if $\sigma(x) = \sigma(y)$.
\end{proposition}

\begin{proof}
The permutations $x=x_1\cdots x_n$ and $y=y_1 \cdots y_n$ are equivalent if and only if $\pi_\downarrow(x) = \pi_\downarrow(y)$.  
Thus to prove the forward direction of the proposition, it is enough to consider the case where $y$ is obtained from $x$ by a single $\pi_\downarrow$-move.  
Consider a $\pi_\downarrow$-move switching $x_i$ and $x_{i+1}$ of $x$.  First we suppose that $x_i \geq x_{i+1}+3$.  
The relative position of $x_i$ with regard to $x_{i+1}$ is irrelevant to the map $\sigma$, thus $\sigma(x) = \sigma(y)$.  
Now we suppose that $x_i = x_{i+1}+2$.  
There can only be a $\pi_\downarrow$-move switching $x_i$ and $x_{i+1}$ of $x$ if $x_{i+1}+1$ is to the left of $x_i$.  
In this case, both $\sigma(x)$ and $\sigma(y)$ have a white square in the $i^{th}$ position and a black square in the $(i+1)^{st}$ position.  Therefore $\sigma(x) = \sigma(y)$.

To prove the reverse implication suppose that $x$ and $y$ are not equivalent, that is \linebreak[4] $\pi_\downarrow(x) \neq \pi_\downarrow(y)$.
Since $\pi_\downarrow(x)$ and $\pi_\downarrow(y)$ are Pell permutations, and $\sigma$ is a bijection from Pell permutations to sashes, $\sigma(\pi_\downarrow(x)) \neq \sigma(\pi_\downarrow(y))$.  
But by the previous paragraph, $\sigma(\pi_\downarrow(x)) = \sigma(x)$ and $\sigma(\pi_\downarrow(y)) = \sigma(y)$.
\end{proof}

The partial order on $\Sigma_{n-1}$ is such that the map $\sigma: P_n \to \Sigma_{n-1}$ is an order isomorphism from the lattice of Pell permutations to $\Sigma_{n-1}$.  
We refer to this lattice as $\Sigma_{n-1}$.

From Proposition~\ref{latcong2.2}, the cover relations in $\Sigma_{n-1}$ are exactly the relations $\sigma(y) \precdot \sigma(x)$ where $x \in P_n$ and $y$ is covered by $x$ in $S_n$.

\begin{proposition}
The cover relations on sashes are 
\begin{enumerate}
\item $A \includesash{-1.5pt}{b} B \precdot A \includesash{-1.5pt}{w} B$ for any sash $A$ and for a sash $B$ whose leftmost tile is not a white square
\item $A \includesash{-1.5pt}{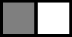} B \precdot A \includesash{-1.5pt}{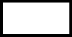} B$ for any sash $A$ and any sash $B$
\item $A \includesash{-1.5pt}{r} B \precdot A \includesash{-1.5pt}{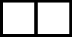} B$ for any sash $A$ and any sash $B$
\end{enumerate}
\end{proposition}

\begin{proof}
Let $x \in P_n$ and let $y \in S_n$ such that $y$ is covered by $x$ in the weak order.  
That is, $x=x_1 \cdots x_i x_{i+1} \cdots x_n$ and $y=x_1 \cdots x_{i+1} x_i \cdots x_n \in S_n$ for some $x_i > x_{i+1}$.

Suppose $x_i= x_{i+1}+1$ and $x_i+1$ is not to the left of $x_i$.
Let $A = \sigma(x|_{[1,x_{i+1}]}) = \sigma(y|_{[1,x_{i+1}]})$ and let $B = \sigma(x|_{[x_i,n]}) = \sigma(y|_{[x_i,n]})$.
Thus, $ A \includesash{-1.5pt}{b} B = \sigma(y) \precdot \sigma(x) = A \includesash{-1.5pt}{w} B$, where the leftmost tile of $B$ is not a white square.

Suppose $x_i= x_{i+1}+1$ and $x_i+1$ is to the left of $x_i$.
Let $A = \sigma(x|_{[1,x_{i+1}]}) = \sigma(y|_{[1,x_{i+1}]})$ and let $B = \sigma(x|_{[x_i+1,n]}) = \sigma(y|_{[x_i+1,n]})$.
Thus, $ A \includesash{-1.5pt}{r} B = \sigma(y) \precdot \sigma(x) = A \includesash{-1.5pt}{ww} B$.

Suppose $x_i= x_{i+1}+2$.
Let $A = \sigma(x|_{[1,x_{i+1}]}) = \sigma(y|_{[1,x_{i+1}]})$ and let $B = \sigma(x|_{[x_i,n]}) = \sigma(y|_{[x_i,n]})$.
Thus, $ A \includesash{-1.5pt}{bw} B = \sigma(y) \precdot \sigma(x) = A \includesash{-1.5pt}{r} B$.
\end{proof}

\begin{example}
See Figure~\ref{fig: Sigma 3&4} for the poset on $\Sigma_3$ and $\Sigma_4$.
\end{example}

\section{The Hopf Algebra (and Dual Hopf Algebra) of Sashes}
\label{HA}

The bijection $\sigma$ allows us to carry the Hopf algebra structure on Pell permutations to a Hopf algebra structure $(\K[\Sigma_\infty], \bullet_{S}, \Delta_{S})$ on sashes and a dual Hopf algebra $(\K[\Sigma_\infty],\Delta_S^*,m_S^*)$ on sashes, where $\K[\Sigma_\infty]$ is a vector space, over a field $\K$, whose basis elements are indexed by sashes.  
In order to do this, we extend $\sigma$ and $\eta$ to linear maps.  
For each grade $n$ of the vector space, the basis elements are represented by the sashes of length $n-1$.
Recall that the sash of length -1 is represented by $\emptyset$, and the sash of length 0 is represented by $\nobox$. 
Let $A$, $B$, and $C$ be sashes.  
Using $\sigma$, we define a product, coproduct, dual product, and dual coproduct of sashes: 

\begin{equation}
\label{sashprod}
m_S(A,B) = A \bullet_S B = \sigma \big( \eta(A) \bullet_{\Av} \eta(B) \big)
\end{equation}

\begin{equation}
\label{sashcoprod}
\Delta_S(C) = (\sigma \otimes \sigma) \big( \Delta_{\Av} ( \eta(C) ) \big)
\end{equation}

\begin{equation}
\Delta^*_S(A \otimes B) = \sigma \Big( \Delta^*_{\Av}\big( \eta(A) \otimes \eta(B) \big) \Big)
\end{equation}

\begin{equation}
\label{eqdualsashcoprod}
m^*_S(C) = (\sigma \otimes \sigma) \big( m^*_{\Av} ( \eta(C) ) \big)
\end{equation}

These operation definitions are somewhat unsatisfying because they require computing the operation in MR.  
That is, calculating a product or coproduct in this way requires mapping sashes to permutations, performing the operations in MR, throwing out the non-avoiders in the result, and then mapping the remaining permutations back to sashes.  
In the rest of this chapter we show how to compute these operations directly in terms of sashes.

\subsection{Product}

\begin{proposition}
\label{product}
The empty sash $\emptyset$ is the identity for the product $\bullet_S$.  
For sashes $A \neq \emptyset$ and $B \neq \emptyset$, the product $A \bullet_S B$ equals:\\
\[\left\{
	\begin{array}{l l}
		\sum \left[ A \includesash{-1.5pt}{b} B, A' \includesash{-1.5pt}{r} B \right] & \text{if } A = A' \includesash{-1.5pt}{b} \\ [.2in]
		\sum \left[ A \includesash{-1.5pt}{b} B, A \includesash{-1.5pt}{w} B \right] & \text{if } A \neq A' \includesash{-1.5pt}{b}\\
	\end{array} 
\right.\]
where $\sum [D, E]$ is the sum of all the sashes in the interval $[D, E]$ on the lattice of sashes.  
\end{proposition}

To clarify, here are some more specific cases of the product on sashes.
If $A = A' \includesash{-1.5pt}{b}$, then $\sum \left[ A \includesash{-1.5pt}{b} B, A' \includesash{-1.5pt}{r} B \right]$ equals:\\
\[\left\{
	\begin{array}{l l}
		A \includesash{-1.5pt}{b} B + A \includesash{-1.5pt}{r} B' + A \includesash{-1.5pt}{w} B + A' \includesash{-1.5pt}{r} B & \text{if } B = \includesash{-1.5pt}{w} B' \\ [.2in]
		A \includesash{-1.5pt}{b} B + A \includesash{-1.5pt}{w} B + A' \includesash{-1.5pt}{r} B & \text{if } B \neq \includesash{-1.5pt}{w} B'\\
	\end{array} 
\right.\]
and if $A \neq A' \includesash{-1.5pt}{b}$, then $\sum \left[ A \includesash{-1.5pt}{b} B, A \includesash{-1.5pt}{w} B \right]$ equals:\\
\[\left\{
	\begin{array}{l l}
		A \includesash{-1.5pt}{b} B + A \includesash{-1.5pt}{r} B' + A \includesash{-1.5pt}{w} B & \text{if } B = \includesash{-1.5pt}{w} B' \\ [.2in]
		A \includesash{-1.5pt}{b} B + A \includesash{-1.5pt}{w} B & \text{if } B \neq \includesash{-1.5pt}{w} B'\\
	\end{array} 
\right.\]

The case where $A=\nobox$ is an instance of $A \neq A' \includesash{-1.5pt}{b}$, and similarly for $B=\nobox$. 

\begin{proof}
We begin by computing the product of Pell permutations.
We showed in Section~\ref{intro} that the product of Pell permutations is the sum over the interval $[x \cdot y', \pi_\downarrow( y' \cdot x ) ]$ in the lattice of Pell permutations, where $x\in P_p$, $y \in P_q$, and $y' =  (y)_{[p+1,n]}$.  

To compute the product of sashes we can apply the map $\sigma$ to the product of Pell permutations.  
Let $\sigma(x)=A$ and $\sigma(y)=B$, thus $A \bullet_S B = \sigma(x \bullet_P y) = \sum [\sigma(x \cdot y'), \sigma( \pi_\downarrow( y' \cdot x ))] = \sum [\sigma(x \cdot y'), \sigma( y' \cdot x ) ]$.
The map $\sigma$ takes the first $p$ values of $x \cdot y'$ to $A$ and the last $q$ values to $B$.  Because $p+1$ is to the right of $p$ and since $p+2$ is not to the left of $p$, $\sigma(x \cdot y') = A \includesash{-1.5pt}{b} B$.
Similarly, $\sigma$ takes the first $p$ values of $y' \cdot x$ to $A$ and the last $q$ values to $B$.  Since $p+1$ is to the left of $p$, to compute $\sigma(y' \cdot x)$ we need to consider whether or not $p-1$ is before $p$ in $x$.

Suppose $p-1$ is before $p$ in $x$.  
Thus, $A$ ends with a black square so $\sigma(y' \cdot x)$ replaces the last black square of $A$ with a rectangle in positions $p-1$ and $p$.
That is $\sigma(y' \cdot x) = A' \includesash{-1.5pt}{r} B$, where $A=A'\includesash{-1.5pt}{b}$.

Suppose $p-1$ is not before $p$ in $x$.  
Thus, $A$ either ends with a white square, the right half of a rectangle, or $p-1$ does not exist. 
Thus, $\sigma(y' \cdot x)$ places a white square after $A$ and before $B$, so $\sigma(y' \cdot x) = A \includesash{-1.5pt}{w} B$.
\end{proof}

In informal terms, the product of two sashes is the sum of the sashes created by joining the two sashes with a black square and a white square, and if by so doing an adjacent black square to the left of a white square is created, then the product has additional terms with rectangles in the places of the adjacent black square and white square.

\begin{example}
Let $A = \includesash{-1.5pt}{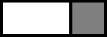}$ and let $B = \includesash{-1.5pt}{w}$.  
Notice that for $A' = \includesash{-1.5pt}{r}$ and $B' = \nobox$, both $A = A' \includesash{-1.5pt}{b}$ and $B = \includesash{-1.5pt}{w} B'$.
\[
\begin{array} {r@{\hspace{2pt}}c@{\hspace{2pt}}c@{\hspace{2pt}}c@{\hspace{2pt}}l@{\hspace{2pt}}c@{\hspace{2pt}}c@{\hspace{2pt}}c@{\hspace{2pt}}c} 
A \bullet_S B & = & \hspace{.29in}A \includesash{-1.5pt}{b} B &+&  \hspace{.31in}A \includesash{-1.5pt}{r} B' &+&  \hspace{.29in}A \includesash{-1.5pt}{w} B &+&  \hspace{.1in}A' \includesash{-1.5pt}{r} B \\ [3pt]
\includesash{-1.5pt}{rb} \bullet_S \includesash{-1.5pt}{w} & = & \includesash{-1.5pt}{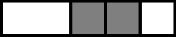} &+& \includesash{-1.5pt}{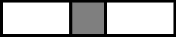} &+& \includesash{-1.5pt}{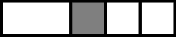} &+& \includesash{-1.5pt}{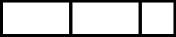} 
\end{array}
\]
\end{example}

\subsection{Dual Coproduct}

From Equation~\eqref{AVdualcoprod} and Equation~\eqref{MRdualcoprod}, it follows that: 
\begin{equation}
\label{dualsashcoprod}
m^*_S(C) = \sum_{i=0}^{n} \sigma\big(\eta(C)|_{[i]}\big) \otimes \sigma\big(\eta(C)|_{[i+1,n]}\big)
\end{equation}

\begin{proposition} The dual coproduct on a sash $C \in \Sigma_n$ is given by: 
\[ m_S^* (C) = \sum_{i=-1}^{n} C_i \otimes C^{n-i-1} \] 
Where $C_i \in \Sigma_i$ is a sash identical to the first $i$ positions of $C$ (unless C has \includesash{-1.5pt}{lhr} in position $i$, in which case $C_i$ ends with \includesash{-1.5pt}{b} ), and  $C^{n-i-1}  \in \Sigma_{n-i-1}$ is a sash identical to the last $n-i-1$ positions of $C$ (unless C has \includesash{-1.5pt}{rhr} in position $i+2$, in which case $C^{n-i-1}$ begins with \includesash{-1.5pt}{w} ), and we define $C_0 = C^0 = \nobox$ and $C_{-1} = C^{-1} = \emptyset$.
\end{proposition}

\begin{proof}
We need to show that $C$ is a term of $A \bullet_S B$ if and only if $A \otimes B$ is a term of $m_S^* (C)$.

Suppose that $C \in \Sigma_n$ is a term of $A \bullet_S B$, with $A \in \Sigma _p$, $B\in \Sigma_q$, and $p+q=n-1$.  
Thus $C$ is one of the following: $A \includesash{-1.5pt}{b} B$, $A \includesash{-1.5pt}{w} B$, $A' \includesash{-1.5pt}{r} B$, or $A \includesash{-1.5pt}{r} B'$, for $A'$ and $B'$ as in Proposition~\ref{product}.  
In any case, $m_S^* (C)$ has a term $A \otimes B$ because $C_p = A$ and $C^{n-p-1} = B$.

Now suppose that for $A\in \Sigma_p$ and $B\in \Sigma_q$, $A \otimes B$ is a term of $m_S^* (C)$, where $C \in \Sigma_n$ and $p+q=n-1$.  
Thus $C_p = A$ and $C^{n-p-1} = C^q = B$.  
If $C$ is $A \includesash{-1.5pt}{b} B$ or $A \includesash{-1.5pt}{w} B$, then $C$ is a term of $A \bullet_S B$.  
If $C$ is $A' \includesash{-1.5pt}{r} B$, then $A = A' \includesash{-1.5pt}{b} $, and $C$ is a term of $A \bullet_S B$.  
If $C$ is $A \includesash{-1.5pt}{r} B'$, then $B = \includesash{-1.5pt}{w} B' $, and $C$ is a term of $A \bullet_S B$.  
Therefore in all cases $C$ is a term of $A \bullet_S B$ and we have shown that the map $m_S^*$ is the dual coproduct on sashes.
\end{proof}

\subsection{Dual Product}

From Equation~\eqref{AVdualprod}, it follows that: 
\begin{equation}
\label{dualsashprod}
\Delta^*_S(A \otimes B) = \sum_{\substack{T\subseteq [n] \\ |T|=p}} \sigma\Big(\big(\eta(A))_T \cdot (\eta(B)\big)_{T^C}\Big)
\end{equation}

We now prepare to describe the dual product $\Delta^*_S$ directly on sashes.

\begin{definition} 
\label{gamma} 
Given a set $T \subseteq [n]$ such that $|T|=p$ and $n=p+q$, and given sashes $D \in \Sigma_{p-1}$ and $E \in \Sigma_{q-1}$, define a sash $\gamma_T(D \otimes E)$ by the following steps. 
First, write $D$ above $E$.  
Then, label $D$ with $T$, by placing the elements of $T$ in increasing order between each position of $D$, including the beginning and end.  
Label $E$ similarly using the elements of $T^C$.  

\begin{example}
\label{ex: dualprod}
Let $T=\{1, 2, 4, 7, 8, 9, 12, 13\}$, $D =\includesash{-1.5pt}{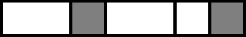}$, and $E = \includesash{-1.5pt}{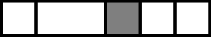}$
\vspace{.1in}
\begin{center}
\scalebox{0.6}{
\begin{picture}(130,70)
\put(0,50){\includegraphics{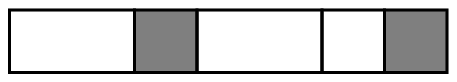}} \put(-1.5,40){\small 1} \put(15,40){\small 2} \put(33,40){\small 4} \put(52,40){\small 7} \put(70,40){\small 8} \put(88,40){\small 9} \put(104,40){\small 12} \put(122,40){\small 13}
\put(9,0){\includegraphics{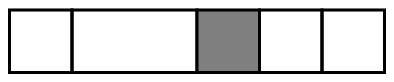}} \put(8,21){\small 3} \put(25,21){\small 5} \put(43,21){\small 6} \put(58,21){\small 10} \put(77,21){\small 11} \put(95,21){\small 14} \put(113,21){\small 15}
\end{picture}}
\end{center}
\end{example}

Next, draw arrows from $i$ to $i+1$ for all $i \in [n-1]$.  
Lastly, follow the path of the arrows placing elements in a new sash based on the following criteria:

Place a rectangle in the $i^{th}$ and $(i+1)^{st}$ positions of the new sash if either of the following conditions are met:
\begin{enumerate}
	\item if the $i^{th}$ arrow is from $D$ to $E$, the $(i+1)^{st}$ arrow is from $E$ to $D$, and there is a \includesash{-1.5pt}{w} or \includesash{-1.5pt}{rhr} in $D$ in between $i$ and $i+2$
	\item if the $i^{th}$ arrow is from $E$ to $E$, the $(i+1)^{st}$ arrow is from $E$ to $D$, and there is a \includesash{-1.5pt}{b} or \includesash{-1.5pt}{lhr} in $E$ in between $i$ and $i+1$
\end{enumerate}
	
If the above criteria are not met, then the following rules apply:
\begin{enumerate}
	\item if the $i^{th}$ arrow is from $D$ to $D$ (or from $E$ to $E$), place whatever is in between $i$ and $i+1$ in $D$ (or in $E$) in the $i^{th}$ position.  
	\item if the $i^{th}$ arrow is from $D$ to $E$, place a black square in the $i^{th}$ position.  
	\item if the $i^{th}$ arrow is from $E$ to $D$, place a white square in the $i^{th}$ position.
\end{enumerate}

Note that it may be necessary to replace the left half of a rectangle by a black square or to replace the right half of a rectangle by a white square (as in the first step of the example below).
\end{definition}

\begin{example} 
Let $T$, $D$, and $E$ be as in Example~\ref{ex: dualprod}.
Then we compute $\gamma_T(D \otimes E)$ to obtain:\\
$\gamma_{\{1, 2, 4, 7, 8, 9, 12, 13\}} ( \includesash{-1.5pt}{rbrwb} \otimes \includesash{-1.5pt}{wrbww} )=\includesash{-1.5pt}{brbrrbrbbw}$.

\begin{center}
\scalebox{0.7}{
\begin{picture}(200,100)

\put(0,80){\includegraphics{large_rbrwb}}
\put(-1.5,70){\small 1}
\put(15,70){\small 2}
\put(33,70){\small 4}
\put(52,70){\small 7}
\put(70,70){\small 8}
\put(88,70){\small 9}
\put(104,70){\small 12}
\put(122,70){\small 13}

\put(9,0){\includegraphics{large_wrbww}}
\put(8,21){\small 3}
\put(25,21){\small 5}
\put(43,21){\small 6}
\put(58,21){\small 10}
\put(77,21){\small 11}
\put(95,21){\small 14}
\put(113,21){\small 15}

\put(0,65){\vector(1,0){15}}
\put(15,62){\vector(-1,-4){7}}
\put(11,32){\vector(1,2){17}}
\put(35,62){\vector(-1,-4){7}}
\put(29,32){\vector(1,0){15}}
\put(46,32){\vector(1,4){8}}
\put(56,65){\vector(1,0){15}}
\put(74,65){\vector(1,0){15}}
\put(90,62){\vector(-1,-1){26}}
\put(64,32){\vector(1,0){15}}
\put(88,32){\vector(1,2){16}}
\put(110,65){\vector(1,0){15}}
\put(127,62){\vector(-1,-1){26}}
\put(101,32){\vector(1,0){15}}

\put(140,45){$\implies$}
\put(172, 42.5){\includegraphics{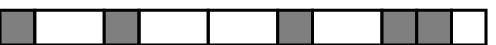}}
\end{picture}
}
\end{center}
\end{example}

\begin{theorem}
\label{Sash dualcoprod}
The dual product of sashes $D \in \Sigma_{p-1}$ and $E \in \Sigma_{q-1}$, for $p+q=n$, is given by:
\[\Delta^*_S(D \otimes E) = \sum_{\substack{T \subseteq [n], \\ |T|=p}} \gamma_T( D\otimes E)\]
\end{theorem}

\begin{proof}
For $D \in \Sigma_{p-1}$ and $E \in \Sigma_{q-1}$ such that $\eta(D)=x \in P_p$ and $\eta(E)=y \in P_q$, where $p+q=n$, we consider Equation~\eqref{dualsashprod} to define the dual product of sashes.

Let $T\subseteq [n]$ such that $|T|=p$.  
It is left to show that $\gamma_T( D\otimes E) = \sigma((x)_T \cdot (y)_{T^C})$.

\textbf{Case 1:} $\gamma_T( D\otimes E)$ has a black square in the $i^{th}$ position.\\
\textbf{Subcase a:} $i, i+1 \in T$ and $D$ has a black square in the position labeled with $i$.\\
\textbf{Subcase b:} $i \in T$, $i+1 \in T^C$, and if $i+2 \in T$ then $D$ does not have a white square or the right half of a rectangle in the position labeled with an $i$.\\
\textbf{Subcase c:} $i, i+1, i+2 \in T^C$ and $E$ has a black square in the position labeled with $i$.\\
For each subcase, the entry $i+1$ in $(x)_T \cdot (y)_{T^C}$ is to the right of $i$ and $i+2$ is not to the left of $i$, so the sash $\sigma((x)_T \cdot (y)_{T^C})$ has a black square in the $i^{th}$ position.\\

\textbf{Case 2:} $\gamma_T( D\otimes E)$ has a white square in the $i^{th}$ position.\\
\textbf{Subcase a:} $i, i+1 \in T$ and $D$ has a white square in the position labeled with $i$.\\
\textbf{Subcase b:} $i-1, i+1 \in T$, $i \in T^C$, and $D$ does not have a white square or the right half of a rectangle in the position labeled with $i-1$.\\
\textbf{Subcase c:} $i+1 \in T$, $i-1, i \in T^C$, and $E$ does not have a black square or the left half of a rectangle in the position labeled with $i-1$.\\
\textbf{Subcase d:} $i, i+1\in T^C$ and $E$ has a white square in the position labeled with $i$.\\
For each subcase, the entry $i+1$ in $(x)_T \cdot (y)_{T^C}$ is to the left of $i$ and $i-1$ is not positioned between $i+1$ and $i$, so the sash $\sigma((x)_T \cdot (y)_{T^C})$ has a white square in the $i^{th}$ position.\\

\textbf{Case 3:} $\gamma_T( D\otimes E)$ has the left half of a rectangle in the $i^{th}$ position.\\
\textbf{Subcase a:} $i, i+1, i+2 \in T$ and $D$ has a rectangle in the positions labeled with $i$ and $i+1$.\\
\textbf{Subcase b:} $i, i+2 \in T$, $i+1 \in T^C$, and $D$ has a white square or the right half of a rectangle in the position labeled with an $i$.\\
\textbf{Subcase c:} $i+2 \in T$, $i, i+1 \in T^C$, and $E$ has a black square or the left half of a rectangle in the position labeled with an $i$.\\
\textbf{Subcase d:} $i, i+1, i+2 \in T^C$ and $E$ has a rectangle in the positions labeled with $i$ and $i+1$.\\
For each subcase, the entry $i+1$ in $(x)_T \cdot (y)_{T^C}$ is to the right of $i$ and $i+2$ is to the left of $i$, so the sash $\sigma((x)_T \cdot (y)_{T^C})$ has the left half of a rectangle in the $i^{th}$ position.\\

\textbf{Case 4:} $\gamma_T( D\otimes E)$ has the right half of a rectangle in the $i^{th}$ position.\\
\textbf{Subcase a:} $i, i+1 \in T$ and $D$ has the right half of a rectangle in the position labeled with $i$.\\
\textbf{Subcase b:} $i-1, i+1 \in T$, $i \in T^C$, and $D$ has a white square or the right half of a rectangle in the position labeled with $i-1$.\\
\textbf{Subcase c:} $i+1 \in T$, $i-1, i \in T^C$, and $E$ has a black square or the left half of a rectangle in the position labeled with $i-1$.\\
\textbf{Subcase d:} $i, i+1\in T^C$ and $E$ has the right half of a rectangle in the position labeled with $i$.\\
For each subcase, the entry $i+1$ in $(x)_T \cdot (y)_{T^C}$ is to the left of $i$ and $i-1$ is positioned between $i+1$ and $i$, so the sash $\sigma((x)_T \cdot (y)_{T^C})$ has the right half of a rectangle in the $i^{th}$ position.\\

Therefore we have shown that $\sigma((x)_T \cdot (y)_{T^C})$ and $\gamma_T( D\otimes E)$ have the same object in every position.
\end{proof}

\subsection{Coproduct}

We now describe the coproduct in the Hopf algebra of sashes and we begin with some definitions.  

\begin{definition}
For $C \in \Sigma_{n-1}$, a \emph{dotting of $C$} is $C$ with a dot in any subset of the $n-1$ positions of $C$.
An \emph{allowable dotting of $C$} is a dotting of $C$ that meets all of the following conditions

\begin{enumerate}
\item has at least one dot
\item the first dot can be in any position, and dotted positions alternate between a black square (or the left half of a rectangle) and a white square (or the right half of a rectangle)
\item has no instances of \includesash{-1.5pt}{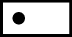}  or \includesash{-1.5pt}{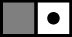} 
\end{enumerate}
\end{definition}

Figure~\ref{allowable dottings} shows the allowable dottings of the sash \includesash{-1.5pt}{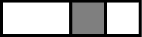}.

\begin{figure}
\begin{center}
\begin{tabular}{cccc}
\includesash{-1.5pt}{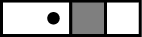} & \includesash{-1.5pt}{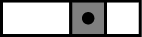} & \includesash{-1.5pt}{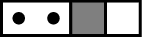} & \includesash{-1.5pt}{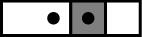} \\[4pt]
\includesash{-1.5pt}{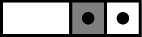} & \includesash{-1.5pt}{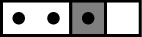} & \includesash{-1.5pt}{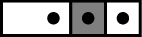} & \includesash{-1.5pt}{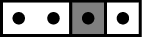} 
\end{tabular}
\end{center}
\caption{The allowable dottings of a sash}
\label{allowable dottings}
\end{figure}

Consider an allowable dotting $d = c_1 \bullet_1 c_2 \bullet_2 \cdots c_j \bullet_j c_{j+1}$ of a sash $C$, where each $c_i$ is a sub sash of $C$ without any dots, and $\bullet_i$ is a single dotted position.  
If any $\bullet_i$ is on the right half of a rectangle, then the the left half of the rectangle in the last position of $c_i$ is replaced by a black square.  
If $\bullet_i$ and $\bullet_{i+1}$ are in adjacent positions, then $c_{i+1}=\nobox$.  
(If any $\bullet_i$ is on the left half of a rectangle, then $\bullet_{i+1}$ is on the right half of the same rectangle, so $c_{i+1}=\nobox$.)  

We use $C$ and $d$ to define two objects $A$ and $B$ that are similar to sashes, but have an additional type of square \includesash{-1.5pt}{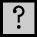}, which we call a mystery square.
If $\bullet_1$ is on a black square or the left half of a rectangle, then let $A$ be the concatenation of the \emph{odd} $c_i$ with a mystery square in between each $c_i$ (where $i$ is odd), and let $B$ be the concatenation of the \emph{even} $c_i$ with a mystery square in between each $c_i$ (where $i$ is even).  
For example, if $\bullet_1$ is on a black square and $j$ is even, then $A = c_1$ \includesash{-1.5pt}{q} $c_3$ \includesash{-1.5pt}{q} $\cdots$ \includesash{-1.5pt}{q} $c_{j+1}$ and $B = c_2$ \includesash{-1.5pt}{q} $c_4$ \includesash{-1.5pt}{q} $\cdots$ \includesash{-1.5pt}{q} $c_j$.
If $\bullet_1$ is on a white square or the right half of a rectangle, then let $A$ be the concatenation of the \emph{even} $c_i$ with a mystery square in between each $c_i$, and let $B$ be the concatenation of the \emph{odd} $c_i$ with a mystery square in between each $c_i$.  

We use the objects $A$ and $B$ to define four sashes $\underline{A}$, $\overline{A}$, $\underline{B}$, and $\overline{B}$.

To compute $\underline{A}$, consider each mystery square in $A$.  
If the mystery square follows $c_i$ and the $i^{th}$ and $(i+1)^{st}$ dots of $d$ are on the same rectangle, \includesash{-1.5pt}{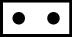}, then replace the mystery square after $c_i$ with a white square.  
Otherwise replace the mystery square with a black square.

To compute $\overline{A}$, consider each mystery square in $A$ from left to right.  
If the mystery square follows $c_i$ and the $i^{th}$ and $(i+1)^{st}$ dots of $d$ are on an adjacent black square and white square, \includesash{-1.5pt}{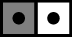}, then we check to see whether or not the mystery square is followed by a white square.
If the mystery square is followed by a white square (\emph{i.e.} if $c_{i+2}$ starts with a white square), then replace the mystery square and the white square with a rectangle.
Otherwise replace the mystery square with a black square.
If the mystery square follows $c_i$ and the $i^{th}$ and $(i+1)^{st}$ dots of $d$ are \emph{not} on an adjacent black square and white square, then we check to see whether or not the mystery square is preceded by a black square.
If either $c_i$ ends in a black square or $c_i = \nobox$ and the previous mystery square has been changed to a black square, then replace the mystery square after $c_i$ and the black square before it with a rectangle.
Otherwise replace the mystery square after $c_i$ with a white square.

To compute $\underline{B}$, replace all mystery squares of $B$ with black squares.

To compute $\overline{B}$, replace all mystery squares of $B$ with white squares, unless the mystery square is preceded by a black square, in which case replace both the black square and the mystery square with a rectangle.

\begin{example} If $d =  \includesash{-1.5pt}{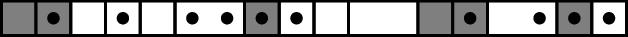} $, then $c_1 = \includesash{-1.5pt}{b}$, $c_2 = c_3 = \includesash{-1.5pt}{w} $, $c_4 = c_5 = c_6 = \nobox$, $c_7 = \includesash{-1.5pt}{wrb}$, $c_8 = \includesash{-1.5pt}{b}$, $c_9 = c_{10} = c_{11} = \nobox$, and $\bullet_1$ is on a black square.  
Thus, $A= c_1$ \includesash{-1.5pt}{q} $c_3$ \includesash{-1.5pt}{q} $c_5$ \includesash{-1.5pt}{q} $c_7$ \includesash{-1.5pt}{q} $c_9$ \includesash{-1.5pt}{q} $c_{11}$ and $B = c_2$ \includesash{-1.5pt}{q} $c_4$  \includesash{-1.5pt}{q} $c_6$  \includesash{-1.5pt}{q} $c_8$  \includesash{-1.5pt}{q} $c_{10}$.  
Using the rules above to compute $A$, $B$, and the four sashes $\underline{A}$, $\overline{A}$, $\underline{B}$, and $\overline{B}$, we have:
\begin{center}
\begin{tabular}{l@{\hskip .3in}l}
$A = \includesash{-1.5pt}{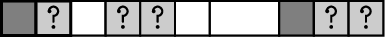}$ & $B = \includesash{-1.5pt}{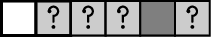}$\\[.07in]
$\underline{A} = $ \includesash{-1.5pt}{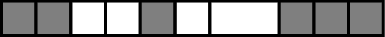} & $\underline{B} = $ \includesash{-1.5pt}{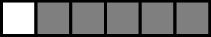}\\[.07in]
$\overline{A} = $ \includesash{-1.5pt}{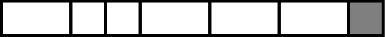} & $\overline{B} = $ \includesash{-1.5pt}{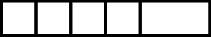}
\end{tabular}
\end{center}
\end{example}

Given an allowable dotting $d$ of a sash $C$ we define $I_d = \sum \left[\underline{A},\overline{A}\right]$ and $J_d=\sum \left[\underline{B},\overline{B}\right]$ for $\underline{A}$, $\overline{A}$, $\underline{B}$, and $\overline{B}$ computed as above.
Thus the notation $I_d \otimes J_d$ denotes $\sum_{\substack{D \in [\underline{A},\overline{A}] \\ E \in [\underline{B},\overline{B}]}} D \otimes E $.

\begin{theorem}
\label{thm:coproduct}
Given $C \in \Sigma_{n-1}$:
\[ \Delta_S(C)= \emptyset \otimes C + C\otimes \emptyset +\sum_{\substack{\text{allowable} \\ \text{dottings} \\ d \text{ of } C}} I_d \otimes J_d \]
\end{theorem}

To prove Theorem~\ref{thm:coproduct}, we need to introduce some more terminology.
Given an allowable dotting $d$ of a sash $C$ and the objects $A$ and $B$ defined above, we define two more objects $\hat{A}$ and $\hat{B}$.  
These objects are similar to sashes, but have three additional types of squares: $\includesash{-1.5pt}{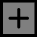}$, $\includesash{-1.5pt}{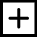}$, and $\includesash{-1.5pt}{q}$.
We call these squares: black-plus square, white-plus square, and mystery square respectively.
If $\bullet_i$ and $\bullet_{i+1}$ are on an adjacent black square and white square, \emph{i.e.} $\includesash{-1.5pt}{bdwd}$, then replace the $\includesash{-1.5pt}{q}$ after $c_i$ on $A$ with a $\includesash{-1.5pt}{bplus}$ on $\hat{A}$.
If $\bullet_i$ and $\bullet_{i+1}$ are on a rectangle, \emph{i.e.} $\includesash{-1.5pt}{rdd}$, then replace the $\includesash{-1.5pt}{q}$ after $c_i$ on $A$ with a $\includesash{-1.5pt}{wplus}$ on $\hat{A}$.
The resulting objects are $\hat{A}$ and $\hat{B}$.
Note $B = \hat{B}$.

We say that a sash $D$ is of the form $\hat{A}$ if $D$ is identical to $\hat{A}$ except for the following allowable substitutions:
\begin{itemize}
\item A black-plus square on $\hat{A}$ is replaced by a black square on $D$.
\item A white-plus square on $\hat{A}$ is replaced by a white square on $D$.
\item A mystery square on $\hat{A}$ is replaced by a either a black square or a white square on $D$.
\item A black-plus square or a mystery square on $\hat{A}$, and a white square, a white-plus square, or a mystery square following it, are replaced by a rectangle on $D$.
\item A white-plus square or a mystery square on $\hat{A}$, and a black square, a black-plus square, or a mystery square preceding it, are replaced by a rectangle on $D$.
\end{itemize}

Similarly, a sash $E$ is of the form $\hat{B}$ if it follows the same rules as above.  Since $\hat{B}$ does not have any black-plus squares or white-plus squares, $E$ is of the form $\hat{B}$ if $E$ is identical to $\hat{B}$ except for the following allowable substitutions:
\begin{itemize}
\item A mystery square on $\hat{B}$ is replaced by a either a black square or a white square on $E$.
\item A mystery square on $\hat{B}$, and a white square or a mystery square following it, are replaced by a rectangle on $E$.
\item A mystery square on $\hat{B}$, and a black square or a mystery square preceding it, are replaced by a rectangle on $E$.
\end{itemize}

\begin{lemma}
The sash $\underline{A}$ is minimal with respect to sashes of the form $\hat{A}$.
\end{lemma}

\begin{proof}
The sash $\underline{A}$ is of the form $\hat{A}$, because every white-plus square on $\hat{A}$ is replaced by a white square on $\underline{A}$ and every black-plus square and mystery square on $\hat{A}$ is replaced by a black square on $\underline{A}$.

Suppose the sash $A'$ is obtained from $\underline{A}$ by going down by a cover relation.
We want to show that $A'$ is not of the form $\hat{A}$.

\textbf{Case 1:} $A' = A_1 \includesash{-1.5pt}{b} A_2$ and $\underline{A} = A_1 \includesash{-1.5pt}{w} A_2$, where the leftmost tile of $A_2$ is not a white square.\\
Let $|A_1| = i - 1$, so $A'$ has a black square in the $i^{th}$ position and $\underline{A}$ has a white square in the $i^{th}$ position.
Thus, $\hat{A}$ either has a white square or a white-plus square in the $i^{th}$ position.
Either way, $A'$ is not of the form $\hat{A}$.

\textbf{Case 2:} $A' = A_1 \includesash{-1.5pt}{bw} A_2$ and $\underline{A} = A_1 \includesash{-1.5pt}{r} A_2$.\\
Let $|A_1| = i - 1$, so $A'$ has a black square and white square in the $i^{th}$ and $(i+1)^{st}$ positions and $\underline{A}$ has a rectangle in the $i^{th}$ and $(i+1)^{st}$ positions.
Thus, $\hat{A}$ has a rectangle in the $i^{th}$ and $(i+1)^{st}$ positions, and $A'$ is not of the form $\hat{A}$.

\textbf{Case 3:} $A' = A_1 \includesash{-1.5pt}{r} A_2$ and $\underline{A} = A_1 \includesash{-1.5pt}{ww} A_2$.\\
Let $|A_1| = i - 1$, so $A'$ has a rectangle in the $i^{th}$ and $(i+1)^{st}$ positions and $\underline{A}$ has a white square in both the $i^{th}$ and $(i+1)^{st}$ positions.
There are four possibilities of what occupies the $i^{th}$ and $(i+1)^{st}$ positions of $\hat{A}$: $\includesash{-1.5pt}{ww}$, $\includesash{-1.5pt}{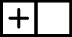}$, $\includesash{-1.5pt}{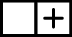}$, or $\includesash{-1.5pt}{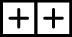}$.
In any case, $A'$ is not of the form $\hat{A}$.
\end{proof}

\begin{lemma}
The sash $\overline{A}$ is maximal with respect to sashes of the form $\hat{A}$.
\end{lemma}

\begin{proof}
The sash $\overline{A}$ is of the form $\hat{A}$, because black-plus squares followed by white-plus squares, mystery squares, or white squares on $\hat{A}$ are replaced by a rectangle on $\overline{A}$, white-plus squares and mystery squares preceded by black-plus squares or black squares on $\hat{A}$ are replaced by a rectangle on $\overline{A}$, all other black-plus squares on $\hat{A}$ are replaced by black squares on $\overline{A}$, and all other white-plus squares and mystery squares on $\hat{A}$ are replaced by white squares on $\overline{A}$.

Suppose the sash $A'$ is obtained from $\overline{A}$ by going up by a cover relation.
We want to show that $A'$ is not of the form $\hat{A}$.

\textbf{Case 1:} $\overline{A} = A_1 \includesash{-1.5pt}{b} A_2$ and $A' = A_1 \includesash{-1.5pt}{w} A_2$, where the leftmost tile of $A_2$ is not a white square.\\
Let $|A_1| = i - 1$, so $\overline{A}$ has a black square in the $i^{th}$ position and $A'$ has a white square in the $i^{th}$ position.
Thus, $\hat{A}$ either has a black square or a black-plus square in the $i^{th}$ position.  
Either way, $A'$ is not of the form $\hat{A}$.

\textbf{Case 2:} $\overline{A} = A_1 \includesash{-1.5pt}{bw} A_2$ and $A' = A_1 \includesash{-1.5pt}{r} A_2$.\\
Let $|A_1| = i - 1$, so $\overline{A}$ has a black square and white square in the $i^{th}$ and $(i+1)^{st}$ positions and $A'$ has a rectangle in the $i^{th}$ and $(i+1)^{st}$ positions.
Thus, $\hat{A}$ has a black square and white square in the $i^{th}$ and $(i+1)^{st}$ positions, and $A'$ is not of the form $\hat{A}$.

\textbf{Case 3:} $\overline{A} = A_1 \includesash{-1.5pt}{r} A_2$ and $A' = A_1 \includesash{-1.5pt}{ww} A_2$.\\
Let $|A_1| = i - 1$, so $\overline{A}$ has a rectangle in the $i^{th}$ and $(i+1)^{st}$ positions and $A'$ has a white square in both the $i^{th}$ and $(i+1)^{st}$ positions.
There are five possibilities of what occupies the $i^{th}$ and $(i+1)^{st}$ positions of $\hat{A}$: $\includesash{-1.5pt}{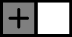}$, $\includesash{-1.5pt}{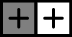}$, $\includesash{-1.5pt}{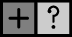}$, $\includesash{-1.5pt}{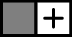}$, or $\includesash{-1.5pt}{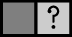}$.
In any case, $A'$ is not of the form $\hat{A}$.
\end{proof}

\begin{lemma}
The sash $\underline{B}$ is minimal with respect to sashes of the form $\hat{B}$.
\end{lemma}

\begin{proof}
The sash $\underline{B}$ is of the form $\hat{B}$, because every mystery square on $\hat{B}$ is replaced by a black square on $\underline{B}$.

Suppose the sash $B'$ is obtained from $\underline{B}$ by going down by a cover relation.
We want to show that $B'$ is not of the form $\hat{B}$.

\textbf{Case 1:} $B' = B_1 \includesash{-1.5pt}{b} B_2$ and $\underline{B} = B_1 \includesash{-1.5pt}{w} B_2$, where the leftmost tile of $B_2$ is not a white square.\\
Let $|B_1| = i - 1$, so $B'$ has a black square in the $i^{th}$ position and $\underline{B}$ has a white square in the $i^{th}$ position.
Thus, $\hat{B}$ has a white square in the $i^{th}$ position, and $B'$ is not of the form $\hat{B}$.

\textbf{Case 2:} $B' = B_1 \includesash{-1.5pt}{bw} B_2$ and $\underline{B} = B_1 \includesash{-1.5pt}{r} B_2$.\\
Let $|B_1| = i - 1$, so $B'$ has a black square and white square in the $i^{th}$ and $(i+1)^{st}$ positions and $\underline{B}$ has a rectangle in the $i^{th}$ and $(i+1)^{st}$ positions.
Thus, $\hat{B}$ has a rectangle in the $i^{th}$ and $(i+1)^{st}$ positions, and $B'$ is not of the form $\hat{B}$.

\textbf{Case 3:} $B' = B_1 \includesash{-1.5pt}{r} B_2$ and $\underline{B} = B_1 \includesash{-1.5pt}{ww} B_2$.\\
Let $|B_1| = i - 1$, so $B'$ has a rectangle in the $i^{th}$ and $(i+1)^{st}$ positions and $\underline{B}$ has a white square in both the $i^{th}$ and $(i+1)^{st}$ positions.
Thus, $\hat{B}$ has a white square in both the $i^{th}$ and $(i+1)^{st}$ positions, and $B'$ is not of the form $\hat{B}$.
\end{proof}

\begin{lemma}
The sash $\overline{B}$ is maximal with respect to sashes of the form $\hat{B}$.
\end{lemma}

\begin{proof}
The sash $\overline{B}$ is of the form $\hat{B}$, because every mystery square on $\hat{B}$ is replaced by a black square on $\overline{B}$, unless it is preceded by a black square, in which case the black square and the mystery square are replaced by a rectangle on $\overline{B}$.

Suppose the sash $B'$ is obtained from $\overline{B}$ by going up by a cover relation.
We want to show that $B'$ is not of the form $\hat{B}$.

\textbf{Case 1:} $\overline{B} = B_1 \includesash{-1.5pt}{b} B_2$ and $B' = B_1 \includesash{-1.5pt}{w} B_2$, where the leftmost tile of $B_2$ is not a white square.\\
Let $|B_1| = i - 1$, so $\overline{B}$ has a black square in the $i^{th}$ position and $B'$ has a white square in the $i^{th}$ position.
Thus, $\hat{B}$ has a black square in the $i^{th}$ position, and $B'$ is not of the form $\hat{B}$.

\textbf{Case 2:} $\overline{B} = B_1 \includesash{-1.5pt}{bw} B_2$ and $B' = B_1 \includesash{-1.5pt}{r} B_2$.\\
Let $|B_1| = i - 1$, so $\overline{B}$ has a black square and white square in the $i^{th}$ and $(i+1)^{st}$ positions and $B'$ has a rectangle in the $i^{th}$ and $(i+1)^{st}$ positions.
Thus, $\hat{B}$ has a black square and white square in the $i^{th}$ and $(i+1)^{st}$ positions, and $B'$ is not of the form $\hat{B}$.

\textbf{Case 3:} $\overline{B} = B_1 \includesash{-1.5pt}{r} B_2$ and $B' = B_1 \includesash{-1.5pt}{ww} B_2$.\\
Let $|B_1| = i - 1$, so $\overline{B}$ has a rectangle in the $i^{th}$ and $(i+1)^{st}$ positions and $B'$ has a white square in both the $i^{th}$ and $(i+1)^{st}$ positions.
Thus, $\hat{B}$ either has a rectangle or a black square and mystery square in the $i^{th}$ and $(i+1)^{st}$ positions.  
Either way, $B'$ is not of the form $\hat{B}$.
\end{proof}

\begin{proposition}
\label{Aint}
If a sash $D$ is of the form $\hat{A}$, then $D \in [\underline{A}, \overline{A}]$.
\end{proposition}

\begin{proof}
Suppose that $D$ is of the form $\hat{A}$ and that $D \neq \underline{A}$.
We want to show that there exists a sash $D'$ such that $D' \precdot D$ and $D'$ is of the form $\hat{A}$.

\textbf{Case 1:} For some white-plus square of $\hat{A}$, it is not replaced by a white square on $D$.\\
Since $D$ is of the form $\hat{A}$, the white-plus square of $\hat{A}$ is preceded by either a black square, a black-plus square, or a mystery square and is replaced by the right half of a rectangle on $D$.
Thus $D = D_1 \includesash{-1.5pt}{r} D_2$.
Let $D' = D_1 \includesash{-1.5pt}{bw} D_2$, so that $D' \precdot D$.
In any case, $D'$ is of the form $\hat{A}$.

\textbf{Case 2:} For some black-plus square of $\hat{A}$, it is not replaced by a black square on $D$.\\
Since $D$ is of the form $\hat{A}$, the black-plus square of $\hat{A}$ is followed by either a white square, a white-plus square, or a mystery square and is replaced by the left half of a rectangle on $D$.
Thus $D = D_1 \includesash{-1.5pt}{r} D_2$.
Let $D' = D_1 \includesash{-1.5pt}{bw} D_2$, so that $D' \precdot D$.
In any case, $D'$ is of the form $\hat{A}$.

\textbf{Case 3:} For some mystery square of $\hat{A}$, it is not replaced by a black square on $D$.

\textbf{Subcase 3a:} The mystery square of $\hat{A}$ is replaced by a white square on $D$.
If $D = D_1 \includesash{-1.5pt}{w} D_2$, where the first tile of $D_2$ is not a white square, then let $D' = D_1 \includesash{-1.5pt}{b} D_2$, so that $D' \precdot D$.
If $D = D_1 \includesash{-1.5pt}{ww} D_3$, then let $D' = D_1 \includesash{-1.5pt}{r} D_3$, so that $D' \precdot D$.
Either way, the sash $D'$ is of the form $\hat{A}$.

\textbf{Subcase 3b:} The mystery square of $\hat{A}$ is replaced by the left half of a rectangle on $D$.\\
Since $D$ is of the form $\hat{A}$, the mystery square of $\hat{A}$ is followed by either a white square, a white-plus square, or another mystery square.
Thus $D = D_1 \includesash{-1.5pt}{r} D_2$.
Let $D' = D_1 \includesash{-1.5pt}{bw} D_2$, so that $D' \precdot D$.
In any case, $D'$ is of the form $\hat{A}$.

\textbf{Subcase 3c:} The mystery square of $\hat{A}$ is replaced by the right half of a rectangle on $D$.\\
Since $D$ is of the form $\hat{A}$, the mystery square of $\hat{A}$ is preceded by either a black square, a black-plus square, or another mystery square.
Thus $D = D_1 \includesash{-1.5pt}{r} D_2$.
Let $D' = D_1 \includesash{-1.5pt}{bw} D_2$, so that $D' \precdot D$.
In any case, $D'$ is of the form $\hat{A}$.

Now, suppose that $D$ is of the form $\hat{A}$ and that $D \neq \overline{A}$.
We want to show that there exists a sash $D'$ such that $D \precdot D'$ and $D'$ is of the form $\hat{A}$.
Consider a black-plus square, white-plus square, or mystery square in the $i^{th}$ position of $\hat{A}$.

\textbf{Case 1:} $D = D_1 \includesash{-1.5pt}{b} D_2$ where $|D_1| = i-1$, and the $(i+1)^{st}$ position of $D$ is not a white square.\\
Let $D' = D_1 \includesash{-1.5pt}{w} D_2$, so that $D \precdot D'$.
If the $i^{th}$ position of $\hat{A}$ is a black-plus square, then $D = \overline{A}$.
So, the $i^{th}$ position of $\hat{A}$ is a mystery square.
The sash $D'$ is of the form $\hat{A}$.

\textbf{Case 2:} $D = D_1 \includesash{-1.5pt}{bw} D_2$ where $|D_1| = i-1$.\\
Let $D' = D_1 \includesash{-1.5pt}{r} D_2$, so that $D \precdot D'$.
The $i^{th}$ position of $\hat{A}$ is either a black-plus square or a mystery square.
The $(i+1)^{st}$ position of $\hat{A}$ is either a white square, a white-plus square, or a mystery square.
Thus, the sash $D'$ is of the form $\hat{A}$.

\textbf{Case 3:} $D = D_1 \includesash{-1.5pt}{r} D_2$ where $|D_1| = i-1$.\\
Let $D' = D_1 \includesash{-1.5pt}{ww} D_2$, so that $D \precdot D'$.
The $i^{th}$ position of $\hat{A}$ is either a black-plus square or a mystery square.
Suppose that the $i^{th}$ position of $\hat{A}$ is a black-plus square.
If the $(i+1)^{st}$ position of $\hat{A}$ is either a white square, a white-plus square, or a mystery square, then $D = \overline{A}$.
If the $(i+1)^{st}$ position of $\hat{A}$ is any other object, then $D$ is not of the form $\hat{A}$.
Thus, the $i^{th}$ position of $\hat{A}$ is a mystery square, and the $(i+1)^{st}$ position of $\hat{A}$ is either a white square, a white-plus square, or a mystery square.
The sash $D'$ is of the form $\hat{A}$.

\textbf{Case 4:} $D = D_1 \includesash{-1.5pt}{r} D_2$ where $|D_1| = i-2$.\\
Let $D' = D_1 \includesash{-1.5pt}{ww} D_2$, so that $D \precdot D'$.
The $i^{th}$ position of $\hat{A}$ is either a white-plus square or a mystery square.
If the $(i-1)^{st}$ position of $\hat{A}$ is either a black square or a black-plus square, then $D = \overline{A}$.
So, the $(i-1)^{st}$ position of $\hat{A}$ is a mystery square, and the sash $D'$ is of the form $\hat{A}$.

\textbf{Case 5:} $D = D_1 \includesash{-1.5pt}{bw} D_2$ where $|D_1| = i-2$, and the $(i-1)^{st}$ position of $\hat{A}$ is a black square.\\
Let $D' = D_1 \includesash{-1.5pt}{r} D_2$, so that $D \precdot D'$.
The $i^{th}$ position of $\hat{A}$ is either a white-plus square or a mystery square.
Thus, the sash $D'$ is of the form $\hat{A}$.
\end{proof}

\begin{proposition}
\label{Bint}
If a sash $E$ is of the form $\hat{B}$, then $E \in [\underline{B}, \overline{B}]$.
\end{proposition}

\begin{proof}
Suppose that $E$ is of the form $\hat{B}$ and that $E \neq \underline{B}$.
We want to show that there exists a sash $E'$ such that $E' \precdot E$ and $E'$ is of the form $\hat{B}$.
If $E \neq \underline{B}$, then for some mystery square of $\hat{B}$, the mystery square is not replaced by a black square on $E$.

\textbf{Case 1:} The mystery square in the $i^{th}$ position of $\hat{B}$ is replaced by a white square on $E$.

\textbf{Subcase 1a:} $E = E_1 \includesash{-1.5pt}{w} E_2$, where the first tile of $E_2$ is not a white square.\\
Let $E' = E_1 \includesash{-1.5pt}{b} E_2$, so that $E' \precdot E$.
The sash $E'$ is of the form $\hat{B}$.

\textbf{Subcase 1b:} $E = E_1 \includesash{-1.5pt}{ww} E_2$.\\
Let $E' = E_1 \includesash{-1.5pt}{r} E_2$, so that $E' \precdot E$, and where $i-1 = |E_1|$.
The $(i+1)^{st}$ position of $\hat{B}$ is either a white square or a mystery square.
Either way, $E'$ is of the form $\hat{B}$.

\textbf{Case 2:} The mystery square of $\hat{B}$ is replaced by the left half of a rectangle on $E$.\\
Thus the mystery square of $\hat{B}$ is followed by a white square or another mystery square, and $E = E_1 \includesash{-1.5pt}{r} E_2$.
Let $E' = E_1 \includesash{-1.5pt}{bw} E_2$, so that $E \precdot E'$.
The sash $E'$ is of the form $\hat{B}$.

\textbf{Case 3:} The mystery square of $\hat{B}$ is replaced by the right half of a rectangle on $E$.\\
Thus the mystery square of $\hat{B}$ is preceded by a black square or another mystery square, and $E = E_1 \includesash{-1.5pt}{r} E_2$.
Let $E' = E_1 \includesash{-1.5pt}{bw} E_2$, so that $E \precdot E'$.
The sash $E'$ is of the form $\hat{B}$.

Now, suppose that $E$ is of the form $\hat{B}$ and that $E \neq \overline{B}$.
We want to show that there exists a sash $E'$ such that $E \precdot E'$ and $E'$ is of the form $\hat{B}$.
Consider a mystery square in the $i^{th}$ position of $\hat{B}$.

\textbf{Case 1:} $E = E_1 \includesash{-1.5pt}{b} E_2$ where $|E_1| = i-1$, and the $(i+1)^{st}$ position of $E$ is not a white square.\\
Let $E' = E_1 \includesash{-1.5pt}{w} E_2$, so that $E \precdot E'$.
The sash $E'$ is of the form $\hat{B}$.

\textbf{Case 2:} $E = E_1 \includesash{-1.5pt}{bw} E_2$ where $|E_1| = i-1$.\\
Let $E' = E_1 \includesash{-1.5pt}{r} E_2$, so that $E \precdot E'$.
The $(i+1)^{st}$ position of $\hat{B}$ is either a white square or a mystery square.
Thus, the sash $E'$ is of the form $\hat{B}$.

\textbf{Case 3:} $E = E_1 \includesash{-1.5pt}{r} E_2$ where $|E_1| = i-1$.\\
Let $E' = E_1 \includesash{-1.5pt}{ww} E_2$, so that $E \precdot E'$.
The $(i+1)^{st}$ position of $\hat{B}$ is either a white square or a mystery square.
Thus, the sash $E'$ is of the form $\hat{B}$.

\textbf{Case 4:} $E = E_1 \includesash{-1.5pt}{r} E_2$ where $|E_1| = i-2$.\\
Let $E' = E_1 \includesash{-1.5pt}{ww} E_2$, so that $E \precdot E'$.
If the $(i-1)^{st}$ position of $\hat{B}$ is a black square, then $E = \overline{B}$.
So, the $(i-1)^{st}$ position of $\hat{B}$ is a mystery square.
Thus, the sash $E'$ is of the form $\hat{B}$.

\textbf{Case 5:} $E = E_1 \includesash{-1.5pt}{bw} E_2$ where $|E_1| = i-2$, and the $(i-1)^{st}$ position of $\hat{B}$ is a black square.\\
Let $E' = E_1 \includesash{-1.5pt}{r} E_2$, so that $E \precdot E'$.
Thus, the sash $E'$ is of the form $\hat{B}$.
\end{proof}

\begin{definition}  
Consider an allowable dotting $d$ of a sash $C \in \Sigma_{n-1}$.  
Place the numbers 1 through $n$ before, after, and in between each of the $n-1$ positions of $C$.  
Let $T$ be the set of numbers such that either the nearest dotted square to their right is a black square (or the left half of a rectangle), or the nearest dotted square to their left of a white square (or the right half of a rectangle).  
We say a set $T$ is an \emph{allowable set} for $C$ if it arises in this way from an allowable dotting of $C$.
\end{definition}

\begin{example} 
For $d = \includesash{-1.5pt}{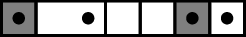}$, the allowable set for $C$ is $T= \{1, 4, 5, 6, 8\}$.
\begin{center}
\begin{picture}(131,35)
\put(0,10){\includegraphics{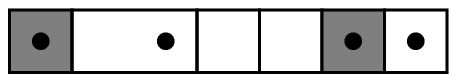}}
\put(-1.5,0){\small \underline{1}}
\put(15,0){\small 2}
\put(33,0){\small 3}
\put(52,0){\small \underline{4}}
\put(70,0){\small \underline{5}}
\put(88,0){\small \underline{6}}
\put(107,0){\small 7}
\put(125,0){\small \underline{8}}
\end{picture}
\end{center}
\end{example}

\begin{definition}
For $T$, an allowable set for $C$, we define $(C)_T$ to be the allowable dotting $d$ of $C$ such that there is a dot in the $i^{th}$ position of $d$ either if $i \in T$ and $i+1 \notin T$ or if $i+1 \in T$ and $i \notin T$.
\end{definition}

\begin{example} 
$( \includesash{-1.5pt}{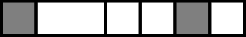} )_{\{1, 4, 5, 6, 8\}} =  \includesash{-1.5pt}{bdrdwwbdwd}$
\end{example}

Figure~\ref{allowable sets and dottings} shows the allowable set associated with each of the allowable dottings in Figure~\ref{allowable dottings}.

\begin{figure}
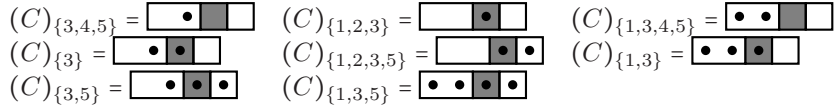

\begin{center}
\begin{tabular}{lll}
$(C)_{\{3,4,5\}} =\includesash{-1.5pt}{rdbw}$ & $(C)_{\{1,2,3\}} = \includesash{-1.5pt}{rbdw}$ & $(C)_{\{1,3,4,5\}} = \includesash{-1.5pt}{rddbw}$\\ 
$(C)_{\{3\}} = \includesash{-1.5pt}{rdbdw}$ & $(C)_{\{1,2,3,5\}} = \includesash{-1.5pt}{rbdwd}$ & $(C)_{\{1,3\}} = \includesash{-1.5pt}{rddbdw}$\\
$(C)_{\{3,5\}} = \includesash{-1.5pt}{rdbdwd}$ & $(C)_{\{1,3,5\}} = \includesash{-1.5pt}{rddbdwd}$ & \\
\end{tabular}
\end{center}
\caption{The allowable sets and allowable dottings of a sash $C$}
\label{allowable sets and dottings}
\end{figure}

We define a map $\tau$ from pairs $(C,T)$ where $C$ is a sash and $T$ is an allowable set with respect to $C$, to permutations.  
The output is a permutation where the elements of $T$ appear before the elements of $T^C$, and as we will verify in Proposition~\ref{tau}, the map $\sigma$ takes $\tau(C,T)$ to $C$.
The map $\tau$ does not necessarily output a Pell permutation.

\begin{definition}  
Let $T$ be an allowable set for a sash $C \in \Sigma_{n-1}$. 
First draw a vertical line.  
We eventually build a permutation by placing all of the elements of $T$ on the left of the vertical line, and all of the elements on $T^C$ on the right of the vertical line, and then removing the line.  
If $1 \in T$ then place a $1$ on the left of the line; otherwise, place the $1$ on the right.  
The guiding principle in defining this map is to place each number $i$ as far right as possible while still making it possible for $\sigma(\tau(C,T))$ to be $C$ and for all entries of $T$ to appear before all entries of $T^C$.
Read the sash from left to right from position 1 to position $n-1$.

Suppose $C$ has a black square in the $i^{th}$ position.
If $i+1 \in T$, then place $i+1$ immediately to the left of the vertical line.
If $i+1 \in T^C$, then place $i+1$ on the far right of the permutation.

Suppose $C$ has a white square in the $i^{th}$ position.
If $i,i+1 \in T$ or if $i,i+1 \in T^C$, then place $i+1$ immediately to the left of $i$.
If $i+1 \in T$ and $i \in T^C$, then place $i+1$ immediately to the left of the vertical line.
The case where $i \in T$ and $i+1 \in T^C$ is ruled out because if there were a dot on the $i^{th}$ position of $C$, which is a white square, then $i \in T^C$ and $i+1 \in T$.

Suppose $C$ has a rectangle in the $i^{th}$ and $(i+1)^{st}$ positions.
All of the possible dottings of the $i^{th}$ and $(i+1)^{st}$ positions of $C$ are shown in Figure~\ref{recdots}.
If $i \in T$, then $i+2 \in T$. 
Place $i+2$ immediately to the left of $i$.
There are two possibilities for placing $i+1$.
If $i+1 \in T$, then place $i+1$ immediately to the left of the vertical line, and if $i+1 \in T^C$, then place $i+1$ at the far right of the permutation.
If $i \in T^C$, then $i+1 \in T^C$.  
Place $i+1$ at the far right of the permutation.
There are two possibilities for placing $i+2$.
If $i+2 \in T$, then place $i+2$ immediately to the left of the vertical line, and if $i+2 \in T^C$, then place $i+2$ immediately to the left of $i$.

Now $\tau(C,T)$ is the permutation that results from ignoring the vertical line.
\end{definition}

\begin{figure}
\begin{center}
\begin{minipage}{.3\textwidth}
	\begin{center}
	\begin{tabular}{c}
	\begin{picture}(40,35)
		\put(0,10){\includegraphics{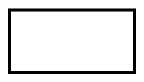}}
		\put(-3,0){\tiny $i$}
		\put(10.5,0){\tiny $i+1$}
		\put(35,0){\tiny $i+2$}
	\end{picture}\\
	$i,i+1,i+2 \in T$\\
	or $i,i+1,i+2 \in T^C$ 
	\end{tabular}
	\end{center}
\end{minipage}
\begin{minipage}{.3\textwidth}
	\begin{center}
	\begin{tabular}{c}
	\begin{picture}(40,35)
		\put(0,10){\includegraphics{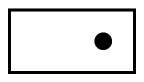}}
		\put(-3,0){\tiny $i$}
		\put(10.5,0){\tiny $i+1$}
		\put(35,0){\tiny $i+2$}
	\end{picture}\\
	$i+2 \in T$\\
	and $i,i+1 \in T^C$ 
	\end{tabular}
	\end{center}
\end{minipage}
\begin{minipage}{.3\textwidth}
	\begin{center}
	\begin{tabular}{c}
	\begin{picture}(40,35)
		\put(0,10){\includegraphics{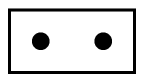}}
		\put(-3,0){\tiny $i$}
		\put(10.5,0){\tiny $i+1$}
		\put(35,0){\tiny $i+2$}
	\end{picture}\\
	$i,i+2 \in T$\\
	and $i+1 \in T^C$ 
	\end{tabular}
	\end{center}
\end{minipage}
\end{center}
\caption{Possible rectangle dottings of the $i^{th}$ and $(i+1)^{st}$ positions of $C$}
\label{recdots}
\end{figure}

\begin{example}
For $C =  \includesash{-1.5pt}{brwwbw}$ and $T= \{ 1, 4, 5, 6, 8 \}$, the following procedure computes $\tau(C,T) = 16548237$.
\begin{center}
\begin{tabular}{llccc}
$1 \in T$& $\rightarrow$ & 1 &\vdots& \\[2pt]
\includesash{-1.5pt}{b} & $\rightarrow$ & 1 &\vdots& 2 \\[2pt]
\includesash{-1.5pt}{r} & $\rightarrow$ & 14 &\vdots& 23 \\[2pt]
\includesash{-1.5pt}{w} & $\rightarrow$ & 154 &\vdots& 23 \\[2pt]
\includesash{-1.5pt}{w} & $\rightarrow$ & 1654 &\vdots& 23 \\[2pt]
\includesash{-1.5pt}{b}& $\rightarrow$ & 1654 &\vdots& 237 \\[2pt]
\includesash{-1.5pt}{w} & $\rightarrow$ & 16548 &\vdots& 237 \\[2pt]
\end{tabular}
\end{center}
\end{example}

\begin{proposition}  
\label{tau}
If $T$ is an allowable set for $C \in \Sigma_{n-1}$, then $\sigma ( \tau(C,T) ) = C$
\end{proposition}

\begin{proof}
Let $T$ be an allowable set for $C$.
Suppose $C$ has a black square in the $i^{th}$ position.  
Since $T$ is allowable, we cannot have $i \in T^C$ and $i+1 \in T$.  
If there is a white square in the $(i+1)^{st}$ position then we cannot have $i, i+1 \in T^C$ and $i+2 \in T$.  
For every situation $\tau(C,T)$ maps $i+1$ to the right of $i$ and does not map $i+2$ to the left of $i$, thus $\sigma ( \tau(C,T) )$ also has a black square in the $i^{th}$ position.

Suppose $C$ has a white square in the $i^{th}$ position.  
Since $T$ is allowable, we cannot have $i \in T$ and $i+1 \in T^C$.  
If there is a black square in the $(i-1)^{st}$ position then we cannot have $i-1, i \in T^C$ and $i+1 \in T$.  
For every situation $\tau(C,T)$ maps $i+1$ to the left of $i$ and, if $i-1$ is to the left of $i$, does not map $i+1$ to the left of $i-1$, thus $\sigma ( \tau(C,T) )$ also has a white square in the $i^{th}$ position.

Suppose $C$ has a rectangle in the $i^{th}$ and $(i+1)^{st}$ positions.  
Considering every possible dotting arrangement as shown in Figure~\ref{recdots}, we see $\tau(C,T)$ maps $i+1$ to the right of $i$ and $i+2$ to the left of $i$, thus $\sigma ( \tau(C,T) )$ also has a rectangle in the $i^{th}$ and $(i+1)^{st}$ positions.
\end{proof}

\begin{proposition} 
$T$ is an allowable set for a sash $C \in \Sigma_{n-1}$, $T=[n]$, or $T=\emptyset$ if and only if $T$ is good as described in Section~\ref{intro}.  
That is, the entries of $T$ are the first elements of a permutation $z' \in S_n$ where $\sigma(z') = C$.
\end{proposition}

\begin{proof}
Suppose $T=[n]$ or $T=\emptyset$, and let $z' = \eta(C)$.  
The entries of $T$ are the first elements of $z'$ such that $\sigma(z') = C$.
Suppose $T$ is an allowable set for a sash $C$, and let $z' = \tau(C,T) \in S_n$.  
The entries of $T$ are the first elements of $z'$ such that $\sigma(z') = C$.

To prove the reverse implication, suppose $z' \in S_n$ such that $\sigma(z') = C$, and let $T_j$ be the set containing the first $j$ entries of $z'$.  
If $j=0$ or $j=n$, then the proposition is true.
Assume that $0<j<n$.
Let $d$ be a dotting of $C$ such that there is a dot in the $i^{th}$ position of $C$ either if $i \in T_j$ and $i+1 \in T^C_j$ or if $i \in T^C_j$ and $i+1 \in T_j$.
We want to show that $T_j$ is an allowable set for $C$; that is, we want to verify that $d$ is an allowable dotting for $C$.  

Since $0<j<n$, there exists some $i$ such that $i \in T_j$ and $i+1 \notin T_j$, so $d$ has at least one dot.

If $i \in T_j$ and $i+1 \in T^C_j$ then $d$ has a dot in the $i^{th}$ position.
We know that entries in $T^C_j$ come after entries of $T_j$ in $z'$, so $i+1$ is to the right of $i$ in $z'$.
Thus the $i^{th}$ dot is on a black square or the left half of a rectangle of $C$.
Similarly, if $i \in T^C_j$ and $i+1 \in T_j$, then the $i^{th}$ dot of $d$ is on a white square or the right half of a rectangle of $C$.
Thus the first dot of $d$ can be on any object, and dotted positions alternate between a black square (or the left half of a rectangle) and a white square (or the right half of a rectangle).

Suppose that there is a rectangle in the $i^{th}$ and $(i+1)^{st}$ positions of $C$.
If $i \in T_j$ and $i+1 \in T^C_j$, then $i+2$ is not an element of $T^C_j$ because $i+2$ is to the left of $i$ in $z'$.
Thus $d$ has no instances of \includesash{-1.5pt}{dr}.

Suppose that there is a black square in the $i^{th}$ position of $C$ and a white square in the $(i+1)^{st}$ position of $C$.
If $i, i+1 \in T^C_j$, then $i+2$ is not an element of $T_j$ because $i+2$ is to the right of $i$ in $z'$.
Thus $d$ has no instances of \includesash{-1.5pt}{bwd}.

Therefore we have shown that $T_j$ is an allowable set for $C$.
\end{proof}

Let $z \in P_n$ and let $C \in \Sigma_{n-1}$ such that $\sigma(z) = C$.
Let $I_T$ and $J_T$ be as defined in Section~\ref{intro}.
From Theorem~\ref{coprod int} we have that the coproduct on sashes is given by:

\begin{equation}
\label{Sash_coprod1}
\Delta_S(C) = \sum_{\substack{T \text{ is allowable,} \\ T=\emptyset \text{, or } T=[n]}} \sigma(I_T)\otimes\sigma(J_T)
\end{equation}

Notice $I_\emptyset = \emptyset$, $J_\emptyset = z$, $I_{[n]} = z$, and $J_{[n]} = \emptyset$, so we have:

\begin{equation}
\label{Sash_coprod}
\Delta_{S}(C) = \emptyset \otimes C + C \otimes \emptyset + \sum_{T \text{ is allowable}} \sigma(I_T)\otimes\sigma(J_T) 
\end{equation}

From Proposition~\ref{coprod dualprod} and our discussion of the relationship between the coproduct and the dual product in Section~\ref{intro},  we see:

\begin{equation}
\label{Sash_coprod2}
\Delta_{S}(C) = \emptyset \otimes C + C \otimes \emptyset + \sum_{T \text{ is allowable}} \, \, \sum_{\substack{D \text{ and } E \text{ such that} \\ \gamma_T(D \otimes E)= C}} D \otimes E 
\end{equation}

\begin{proposition} 
\label{prop:dualprodsashes}  
Given an allowable set $T$ and a sash $C$ such that $(C)_T=d$, $D \otimes E$ is a term of $I_d \otimes J_d$ if and only if $\gamma_T(D \otimes E) =C$.
\end{proposition}

\begin{proof}
Recall the notation $d = c_1 \bullet_1 c_2 \bullet_2 \cdots c_j \bullet_j c_{j+1}$.  
Assume that $j$ is even and that $\bullet_1$ is on a black square or the left side of a rectangle, thus $1 \in T$.  
The cases where $j$ is odd or where $\bullet_1$ is on a white square or the right side of a rectangle are identical, other than some adjustments to indices.
Let $T_{2i-1}$ be the $i^{th}$ set of consecutive integers in $T$ and let $T_{2i}$ be the $i^{th}$ set of consecutive integers in $T^{C}$.  
Thus $T = T_1 \cup T_3 \cup \dots \cup T_{j+1}$ and $T^C = T_2 \cup T_4 \cup \dots \cup T_j$, where $1 \in T_1 \subseteq T$.
Recall $A = c_1$ \includesash{-1.5pt}{q} $c_3$ \includesash{-1.5pt}{q} $\cdots$ \includesash{-1.5pt}{q} $c_{j+1}$ and $B = c_2$ \includesash{-1.5pt}{q} $c_4$ \includesash{-1.5pt}{q} $\cdots$ \includesash{-1.5pt}{q} $c_j$.  
Notice that $c_i = \sigma(\eta(C)|_{T_i})$.  

If $h,h+1 \in T$, then for some index $2i-1$ we have $h,h+1 \in T_{2i-1}$.  
Thus, the $h^{th}$ position of $C$ is in $c_{2i-1}$ and we refer to that position on $A$ as the position \emph{corresponding} to $h$.
Similarly, if $h,h+1 \in T_{2i} \subseteq T^C$, then the $h^{th}$ position of $C$ is in $c_{2i}$ and we refer to that position on $B$ as the position \emph{corresponding} to $h$.

Suppose that $D \otimes E$ is a term of $I_d \otimes J_d$, that is $D \in [\underline{A}, \overline{A}]$ and $E \in [\underline{B}, \overline{B}]$.

As we compute $\gamma_T(D \otimes E)$, we begin by labeling $D$ with the elements of $T$ and labeling $E$ with the elements of $T^C$.  
Notice that the $h^{th}$ arrow meets one of the following conditions:
\begin{enumerate}
\item If $h,h+1 \in T$, then the $h^{th}$ arrow is in the position of $D$ corresponding to $h$.
\item If $h,h+1 \in T^C$, then the $h^{th}$ arrow is in the position of $E$ corresponding to $h$.
\item If $h \in T$ and $h+1 \in T^C$, then the $h^{th}$ arrow is from $D$ to $E$. 
\item If $h \in T^C$ and $h+1 \in T$, then the $h^{th}$ arrow is from $E$ to $D$. 
\end{enumerate}

\textbf{Case 1:} $h,h+1, h+2 \in T$.\\
Whatever is in the $h^{th}$ position of $C$ is also in the position of $D$ corresponding to $h$.  
The map $\gamma_T(D \otimes E)$ places whatever is in the position of $D$ corresponding to $h$ in the $h^{th}$ position of the output.

\textbf{Case 2:} $h,h+1, h+2 \in T^C$.\\
Whatever is in the $h^{th}$ position of $C$ is also in the position of $E$ corresponding to $h$.  
The map $\gamma_T(D \otimes E)$ places whatever is in the position of $E$ corresponding to $h$ in the $h^{th}$ position of the output.

\textbf{Case 3:} $h,h+1 \in T$ and $ h+2 \in T^C$.\\
If $C$ has a rectangle in the $h^{th}$ and $(h+1)^{st}$ positions, then the dotting $d$ would have a dot in the $(h+1)^{st}$ position, which is the right half of a rectangle.
The allowable set associated with such a dotting has $h,h+1 \in T^C$ and $h+2\in T$ which is a contradiction, so $C$ does not have a rectangle in the $h^{th}$ and $(h+1)^{st}$ positions.

If whatever is in the $h^{th}$ position of $C$ is also in the position of $D$ corresponding to $h$, then $\gamma_T(D \otimes E)$ places whatever is in the position of $D$ corresponding to $h$ in the $h^{th}$ position of the output.

If $C$ has a black square in the $h^{th}$ position, then $D$ may have the left half of a rectangle in the position corresponding to $h$.
The map $\gamma_T(D \otimes E)$ places a black square in the $h^{th}$ position of the output.

\textbf{Case 4:} $h,h+1 \in T^C$ and $ h+2 \in T$.\\
If $C$ has a black square in the $h^{th}$ position, then the dotting $d$ would have an instance of \includesash{-1.5pt}{bwd}, which is a contradiction.
Thus, $C$ does not have a black square in the $h^{th}$ position.

If $C$ has a white square or the right half of a rectangle in the $h^{th}$ position, then $E$ has the same object in the position corresponding to $h$.
The map $\gamma_T(D \otimes E)$ places either a white square or the right half of a rectangle respectively in the $h^{th}$ position of the output and a white square in the $(h+1)^{st}$ position.

If $C$ has the left half of a rectangle in the $h^{th}$ position, then $E$ either has a black square or left half of a rectangle in the position corresponding to $h$.
The map $\gamma_T(D \otimes E)$ places a rectangle in the $h^{th}$ and $(h+1)^{st}$ positions of the output.

\textbf{Case 5:} $h \in T$ and $h+1, h+2 \in T^C$.\\
If $C$ has a rectangle in the $h^{th}$ and $(h+1)^{st}$ positions, then the dotting $d$ would have an instance of \includesash{-1.5pt}{dr}, which is a contradiction.
Thus, $C$ does not have a rectangle in the $h^{th}$ and $(h+1)^{st}$ positions and $C$ does have a black square in the $h^{th}$ position.
The map $\gamma_T(D \otimes E)$ places a black square in the $h^{th}$ position of the output.

\textbf{Case 6:} $h, h+2 \in T$ and $h+1 \in T^C$.\\
If $C$ has a black square in the $h^{th}$ position and a white square in the $(h+1)^{st}$ position, then $\underline{A}$ has a black square in the position between the labels $h$ and $h+2$, and $\overline{A}$ has either a black square or the left half of a rectangle in the position between the labels $h$ and $h+2$.
Thus, $D$ has either a black square or the left half of a rectangle in the position between the labels $h$ and $h+2$.
The map $\gamma_T(D \otimes E)$ places a black square in the $h^{th}$ position and a white square in the $(h+1)^{st}$ position of the output.

If $C$ has a rectangle in the $h^{th}$ and $(h+1)^{st}$ positions, then $\underline{A}$ has a white square in the position between the labels $h$ and $h+2$, and $\overline{A}$ has either a white square or the right half of a rectangle in the position between the labels $h$ and $h+2$.
Thus, $D$ has either a white square or the left half of a rectangle in the position between the labels $h$ and $h+2$.
The map $\gamma_T(D \otimes E)$ places a rectangle in the $h^{th}$ and $(h+1)^{st}$ positions of the output.

\textbf{Case 7:} $h \in T^C$ and $h+1, h+2 \in T$.\\
$C$ either has a black square or the left half of a rectangle in the $h^{th}$ position, and whatever is in the $(h+1)^{st}$ position of $C$ is also in the position of $D$ corresponding to $h+1$.
We see from cases 4 and 6 that $\gamma_T(D \otimes E)$ places either a black square or the left half of a rectangle, respectively, in the $h^{th}$ position of the output.
Also, $\gamma_T(D \otimes E)$ places whatever is in the position of $D$ corresponding to $h+1$ in the $(h+1)^{st}$ position of the output.

\textbf{Case 8:} $h, h+2 \in T^C$ and $h+1 \in T$.\\
We see from cases 4 and 6 that whatever is in the $h^{th}$ position of $C$ is also in the $h^{th}$ position of $\gamma_T(D \otimes E)$, and we see from cases 5 and 6 that whatever is in the $(h+1)^{st}$ position of $C$ is also in the $(h+1)^{st}$ position of $\gamma_T(D \otimes E)$.

Therefore we have shown that $\gamma_T(D \otimes E)=C$.

Now let us suppose $\gamma_T(D \otimes E)=C$, and we will show that $D \otimes E$ is a term of $I_d \otimes J_d$.
It is enough to show that $D$ is of the form $\hat{A}$ and that $E$ is of the form $\hat{B}$ because of Proposition~\ref{Aint} and Proposition~\ref{Bint}.
We refer to the position of $D$ or $E$ that is labeled with the $h^{th}$ arrow as the position of $D$ or $E$ corresponding to $h$.

\textbf{Case 1:} $h,h+1, h+2 \in T$.\\
The object in the position of $D$ corresponding to $h$ is the same as the object in the $h^{th}$ position of $C$, which is also the same as the object in the position of $\hat{A}$ corresponding to $h$.

\textbf{Case 2:} $h,h+1, h+2 \in T^C$.\\
The object in the position of $E$ corresponding to $h$ is the same as the object in the $h^{th}$ position of $C$, which is also the same as the object in the position of $\hat{B}$ corresponding to $h$.

\textbf{Case 3:} $h,h+1 \in T$ and $ h+2 \in T^C$.\\
As we showed above, $C$ does not have a rectangle in the $h^{th}$ and $(h+1)^{st}$ positions.  
If $D$ has the left half of a rectangle in the position corresponding to $h$, then $C$ has a black square in the $h^{th}$ position.
Thus, the sash $\hat{A}$ has a black square in the position corresponding to $h$ and a mystery square in the following position.
If $D$ has any other object in the position corresponding to $h$, then $C$ has the same object in the $h^{th}$ position and the sash $\hat{A}$ has the same object as $D$ in the position corresponding to $h$.

\textbf{Case 4:} $h,h+1 \in T^C$ and $ h+2 \in T$.\\
As we showed above, $C$ does not have a black square in the $h^{th}$ position.
If $E$ has a white square or the right half of a rectangle in the position corresponding to $h$, then $C$ has the same object as $E$ in the $h^{th}$ position and a white square in the $(h+1)^{st}$ position.
Thus, $\hat{B}$ has the same object as $E$ in the position corresponding to $h$ followed by a mystery square.
If $E$ has a black square or the left half of a rectangle in the position corresponding to $h$, then $C$ has a rectangle in the $h^{th}$ and $(h+1)^{st}$ positions.
Thus, $\hat{B}$ has a black square in the position corresponding to $h$ followed by a mystery square.

\textbf{Case 5:} $h \in T$ and $h+1, h+2 \in T^C$.\\
If there is a black square in the position of $E$ corresponding to $h+1$, then the $(h+1)^{st}$ position of $C$ is either a black square or the left half of a rectangle.
If $h+3 \in T$, then $C$ has a rectangle in the $(h+1)^{st}$ and $(h+2)^{nd}$ positions, and the sash $\hat{B}$ has a black square in the position corresponding to $h+1$ followed by a mystery square.
If $h+3 \notin T$, then $C$ has a black square in the $(h+1)^{st}$ position, and the sash $\hat{B}$ has a black square in the position corresponding to $h+1$.

If there is the left half of a rectangle in the position of $E$ corresponding to $h+1$, then the $(h+1)^{st}$ position of $C$ is the left half of a rectangle.
If $h+3 \in T$, then the sash $\hat{B}$ has a black square in the position corresponding to $h+1$ followed by a mystery square.
If $h+3 \notin T$, then the sash $\hat{B}$ has a the left half of a rectangle in the position corresponding to $h+1$.

If there is either a white square or the right half of a rectangle in the position of $E$ corresponding to $h+1$, then the $(h+1)^{st}$ position of $C$ is a white square.
The sash $\hat{B}$ has a white square in the position corresponding to $h+1$ preceded by a mystery square. 

\textbf{Case 6:} $h, h+2 \in T$ and $h+1 \in T^C$.\\
If $D$ has a black square in the position between the labels $h$ and $h+2$, then $C$ has a black square in the $h^{th}$ position and a white square in the $(h+1)^{st}$ position.
The sash $\hat{A}$ has a black-plus square in the position between the labels $h$ and $h+2$.

If $D$ has the left half of a rectangle in the position between the labels $h$ and $h+2$, then $C$ has a black square in the $h^{th}$ position and a white square in the $(h+1)^{st}$ position.
If $h+3 \in T$, then the sash $\hat{A}$ has a black-plus square in the position between the labels $h$ and $h+2$ followed by a white square.
If $h+3 \notin T$, then the sash $\hat{A}$ has a black-plus square in the position between the labels $h$ and $h+2$ followed by either a white square, a white-plus square, or a mystery square.

If $D$ has a white square in the position between the labels $h$ and $h+2$, then $C$ has a rectangle in the $h^{th}$ and $(h+1)^{st}$ positions.
The sash $\hat{A}$ has a white-plus square in the position between the labels $h$ and $h+2$.

If $D$ has the right half of a rectangle in the position between the labels $h$ and $h+2$, then $C$ has a rectangle in the $h^{th}$ and $(h+1)^{st}$ positions.
The sash $\hat{A}$ has a white-plus square in the position between the labels $h$ and $h+2$, preceded by either a black square, a black-plus square, or a mystery square..

\textbf{Case 7:} $h \in T^C$ and $h+1 \in T$.\\
This case has already been fully considered in cases 4 and 6.

We have shown, by checking every position, that $D$ is of the form $\hat{A}$ and that $E$ is of the form $\hat{B}$.  
Therefore, $D \in [\underline{A}, \overline{A}]$, $E \in [\underline{B}, \overline{B}]$, and $D \otimes E$ is a term of $I_d \otimes J_d$.
\end{proof}

Theorem~\ref{thm:coproduct} follows directly from Equation~\eqref{Sash_coprod2} and Proposition~\ref{prop:dualprodsashes}.

\section*{Acknowledgments}
This research is a compilation of a portion of the author's doctoral dissertation~\cite{Law13}.
The author would like to acknowledge Nathan Reading for giving generously of his time and sharing his experience while advising this work.
An extended abstract of this work has been published in the proceedings of the 26th International Conference on Formal Power Series and Algebraic Combinatorics (FPSAC 2016)~\cite{LawFPSAC14}.

\bibliography{sash}{}

\begin{thebibliography}{1}

\bibitem{AgBSottile}
Marcelo Aguiar, Nantel Bergeron, and Frank Sottile.
\newblock Combinatorial {H}opf algebras and generalized {D}ehn-{S}ommerville
  relations.
\newblock {\em Compos. Math.}, 142(1):1--30, 2006.

\bibitem{permpat}
Eric Babson and Einar Steingr{\'{\i}}msson.
\newblock Generalized permutation patterns and a classification of the
  {M}ahonian statistics.
\newblock {\em S\'em. Lothar. Combin.}, 44:Art. B44b, 18 pp. (electronic),
  2000.

\bibitem{Law13}
Shirley Law.
\newblock {\em {Combinatorial Realization of Certain Hopf Algebras of
  Pattern-Avoiding Permutations}}.
\newblock PhD thesis, North Carolina State University, 2013.

\bibitem{LawFPSAC14}
Shirley Law.
\newblock Combinatorial realization of the hopf algebra of sashes.
\newblock In {\em Proceedings of the 26th International Conference on Formal
  Power Series and Algebraic Combinatorics (FPSAC 2016)}, July 2014.

\bibitem{LR}
Jean-Louis Loday and Mar{\'{\i}}a~O. Ronco.
\newblock Hopf algebra of the planar binary trees.
\newblock {\em Adv. Math.}, 139(2):293--309, 1998.

\bibitem{mr}
Claudia Malvenuto.
\newblock Produits et coproduits des fonctions quasi-sym\'etriques et de
  l'alg\`ebre des descentes.
\newblock {\em (LACIM)}, 16, 1994.

\bibitem{latcong}
Nathan Reading.
\newblock Lattice congruences, fans and {H}opf algebras.
\newblock {\em J. Combin. Theory Ser. A}, 110(2):237--273, 2005.

\bibitem{OEIS}
N.~J.~A. Sloane.
\newblock The on-line encyclopedia of integer sequences.
\newblock \url{http://oeis.org}.
\newblock 2013.

\end{thebibliography}
\bibliographystyle{plain}

\end{document}